\documentclass[12pt]{article}
\usepackage{amssymb}
\usepackage{amsmath}
\usepackage{amsfonts}
\usepackage{tikz}
\usepackage{bookmark}

\newtheorem{theorem}{Theorem}

\newtheorem{corollary}[theorem]{Corollary}

\newtheorem{lemma}[theorem]{Lemma}

\newtheorem{proposition}[theorem]{Proposition}
\newtheorem{remark}[theorem]{Remark}

\newenvironment{proof}[1][Proof]{\noindent\textbf{#1.} }{\ \rule{0.5em}{0.5em}}
\begin{document}

\title{Moment free deviation inequalities for linear combinations of independent random variables with power-type tails}
\author{Daniel J. Fresen\thanks{%
University of Pretoria, Department of Mathematics and Applied Mathematics,
daniel.fresen@up.ac.za}}
\date{}
\maketitle

\begin{abstract}
We present order of magnitude estimates for the quantiles of non-negative linear combinations of non-negative random variables, as well as deviation inequalities for general linear combinations of independent random variables, under the assumption that all random variables satisfy the same power-type tail bound on $\mathbb{P}\{\left\vert X_i\right\vert>t\}$ of the form $t^{-q}$, $t^{-q/2}$ or $t^{-q/2}(\ln t)^{q/2}$, for $q>2$. The third type is applicable in the nonlinear setting. In the situations we consider, these results improve on classical estimates of Nagaev.
\end{abstract}

\tableofcontents

\section{Introduction}

In Lata\l a's 1997 paper \cite{Lat97}, the problem of estimating $L_p$ norms of sums of independent random variables, i.e. $\left(\mathbb{E}\left\vert \sum_{i=1}^n X_i\right\vert^p\right)^{1/p}$, was reduced to the problem of evaluating a type of Orlicz norm. In the case where the distribution of each $X_i$ decays quicker than any power function, these moment estimates can often be used in conjunction with Markov's inequality to obtain correct order of magnitude bounds on the quantiles of $\sum X_i$. In the setting where the tail probabilities decay like power functions, e.g. $\mathbb{P}\{\left\vert X_i\right\vert >t\}=(1+t)^{-q}$ for $q\in(2,\infty)$, the corresponding $L_p$ norms are finite only in a bounded range of $p$ and order of magnitude estimates on these norms do not contain enough information to recover correct order of magnitude estimates for the quantiles and tail probabilities of $\sum X_i$. In the example just given, one misses sub-Gaussian estimates in the central region of the distribution and is off by a poly-logarithmic factor in the tails.

Deviation inequalities for sums of heavy tailed random variables have been studied extensively, and results are given at varying degrees of precision, generality and usability, under various assumptions on the tails, in the asymptotic sense with $n\rightarrow\infty$ and in the non-asymptotic sense (i.e. quantitative bounds that hold for all $n$ or for $n>n_0$). We refer the reader to \cite{Fr20, GaRaRe, MikNag, NagAop} and the references therein for more details. Of particular relevance is a result of Nagaev (following an earlier result of Linnik), see for example \cite[Theorem 1.9]{NagAop}, which we simplify for convenience: If $q>2$ and $(X_i)_1^n$ are symmetric and i.i.d. with $\mathbb{P}\{\left\vert X_i \right\vert >t\}=(1+t)^{-q}$ for $t>0$, then
\[
\mathbb{P}\left\{\frac{1}{\sqrt{n\mathbb{E}X_i^2}}\sum_{i=1}^nX_i>t\right\}=(1+o(1))\left(1-\Phi(t)\right)+(1+o(1))n\mathbb{P}\{X_i >\sqrt{n}t\}
\]
for $n\rightarrow\infty$ and $t\geq\sqrt{n}$. In the full statement the symmetry is not needed and the CDF may involve a slowly varying function. See the given reference for details. In the case of linear combinations of these same variables, one has
\begin{equation*}
\mathbb{P}\left\{ \sum_{i=1}^na_iX_i>s\right\} \leq \left(1+\frac{2}{p}\right)^p\mathbb{E}\left(\max\{0,X_1\}^p\right)s^{-p}\sum_{i=1}^n\left\vert a_i\right\vert ^p+\exp\left(\frac{-2(p+2)^{-2}e^{-p}s^2}{\sum_{i=1}^n\left\vert a_i\right\vert ^2\mathbb{E}X_1^2}\right)
\end{equation*}
Here we are using \cite[Corollary 1.8]{NagAop} where our $a_iX_i$ represents his $X_i$, and $p\in[2,q)$. This goes back to a 1971 result of Fuk and Nagaev \cite{FuNa71} and implies
\begin{equation*}
\mathbb{P}\left\{\sum_{i=1}^na_iX_i> \frac{p+2}{2}e^{p/2}\left(\mathbb{E}X_1^2\right)^{1/2}t\left\vert a \right\vert_2 +\left[\mathbb{E}\left(\max\{0,X_1\}^p\right)\right]^{1/p}e^{t^2/(2p)}\left\vert a \right\vert_p)\right\}\leq Ce^{-t^2/2}
\end{equation*}
where $\left\vert a \right\vert_q=\left( \sum_{i=1}^{n}\left\vert a_{i}\right\vert^q\right) ^{1/p}$ is the $\ell_q^n$ norm. For large values of $t$ the second term dominates, and to minimize the deviation requires $p$ close to $q$, but not too close. For $p>(q-1)/2$,
\[
\left[\mathbb{E}\left(\max\{0,X_1\}^p\right)\right]^{1/p}\leq\left(\frac{q}{2}\int_0^1u^pdu+\frac{q}{2}\int_1^\infty u^{-q-1+p}du\right)^{1/p}\leq\left(\frac{q}{q-p}\right)^{1/p}
\]
with a lower bound if we include an extra factor of $2^{-(2+q)/p}$. The minimum of
\[
e^{h(p)}:=\exp\left(\frac{t^2}{2p}+\frac{1}{p}\ln \frac{q}{q-p}\right)
\]
for $p\in[2,q)$ is achieved when
\[
\frac{t^2}{2}=\ln\left(1-\frac{1}{s}\right)+\frac{1}{s-1} \hspace{0.5cm}\text{where}\hspace{0.5cm} s=\frac{q}{p}
\]
The solution satisfies
\[
1+\frac{1}{t^2/2+\ln t^2}\leq s\leq 1+\frac{1}{t^2/2+\ln(t^2/2)}
\]
so for large values of $t$ one has
\begin{equation}
\mathbb{P}\left\{\sum_{i=1}^na_iX_i>Ct^{2/q}e^{t^2/(2q)}\left\vert a \right\vert_q)\right\}\leq Ce^{-t^2/2}\label{Nagaev orig}
\end{equation}

The contributions of this paper are twofold:

\medskip

\noindent $i.$ An improved deviation inequality for linear combinations of independent random variables with power type decay, presented in Section \ref{main st}, that removes the factor $t^{2/q}$ from \ref{Nagaev orig}.

\medskip

\noindent $ii.$ Order of magnitude estimates for non-negative linear combinations of non-negative random variables, and various other tools, that are used to prove the deviation inequality mentioned in $(i)$ above. These are also used in \cite{FrIIb} to prove a deviation inequality in the nonlinear setting. We must stress that for us, at least, and we hope for others, these tools are of significant independent interest and utility. They are presented in Section \ref{Theory}.

\section{Main result\label{main st}}

\begin{theorem}
\label{poly tails Lip}There exists a universal constant $C>0$ such that the
following is true. Let $n\in \mathbb{N}$, $2<q<\infty $, let $a\in\mathbb{R}^n$ with $a\neq 0$, and let $\left( X_{i}\right) _{1}^{n}$ be a sequence of independent random variables such that for all $t>0$,%
\begin{equation}
\mathbb{P}\left\{\left\vert X_i \right\vert >t\right\}\leq 2(1+t)^{-q}\label{tail assumption bound}
\end{equation}
For all $t>0$,
\begin{equation}
\mathbb{P}\left\{ \left\vert \sum_{i=1}^na_iX_i -\mathbb{E}\sum_{i=1}^na_iX_i
\right\vert >C_{q}\left(t\left\vert a \right\vert_2 +e^{t^2/(2q)}\left\vert a \right\vert_q\right)\right\} \leq Ce^{-t^2/2}  \label{poly conc one}
\end{equation}
where $C_{q}>1$ is a function of $q$.
\end{theorem}

\underline{Notation and conventions:} $\mathbb{M}$ denotes median, $C$, $c$ etc. denote positive universal constants that may take on different values at each appearance, whose values we do not necessarily control. $C_q$, $c_q$ etc. denote `constants' that depend on $q$ (i.e. functions of $q$). $\left\vert a \right\vert_2$ is often denoted $\left\vert a \right\vert$. Our usage of the term `random variable' is limited to the real valued case.

\section{\label{Theory}Non-negative linear combinations of non-negative i.i.d. RV}

\subsection{\label{SDKW aaa}Sums of order statistics (case of equal coefficients)}

Concentration inequalities for the binomial distribution and the study of order statistics of uniform $(0,1)$ random variables are of course quite standard. See for example \cite{BLMass, ScWe}. In this section we present several results, based on classical techniques like the exponential moment method and the R\'{e}nyi representation of
order statistics, tailored to our purposes.

Define $\xi _{1}:\left[ 0,1\right] \rightarrow \left[
0,1\right] $ and $\xi _{2}:\left[ 0,\infty \right) \rightarrow \left( 0,1%
\right] $ by%
\begin{equation}
\xi _{1}(t)=e^{t}\left( 1-t\right) \hspace{0.65in}\xi _{2}(t)=e^{-t}\left(
1+t\right)  \label{deviation fncts}
\end{equation}%

\begin{lemma}\label{numebecali}
\begin{eqnarray*}
\xi _{1}^{-1}(t) &\leq &\min \left\{ \sqrt{2\left( 1-t\right) }%
,1-e^{-1}t\right\} :0\leq t\leq 1 \\
\xi _{2}^{-1}(t) &\leq &\left\{ 
\begin{array}{ccc}
\log t^{-1}+\log \left( 1+4\log t^{-1}\right) & : & 0<t\leq 2e^{-1} \\ 
\sqrt{2\log t^{-1}+10\left( \log t^{-1}\right) ^{3/2}} & : & 2e^{-1}\leq
t\leq 1%
\end{array}%
\right.
\end{eqnarray*}
\end{lemma}

\begin{proof}
The estimates for $\xi _{1}^{-1}$ follow since $\xi _{1}(t)\leq \min \left\{
1-t^{2}/2,e\left( 1-t\right) \right\} $. To estimate $\xi _{2}^{-1}$ we
re-write $y=e^{-t}\left( 1+t\right) $ as $z=t-\log \left( 1+t\right) $,
where $z=\log y^{-1}$. If $z<1-\log 2$ then $t<1$, since $t\mapsto t-\log
\left( 1+t\right) $ is strictly increasing. Since $\log
(1+t)=\sum_{1}^{\infty }\left( -1\right) ^{j+1}j^{-1}t^{j}$ is alternating,
with terms that decrease in absolute value, $z=t-\log \left( 1+t\right) \geq
t-\left( t-t^{2}/2+t^{3}/3\right) \geq t^{2}/6$. But then $z=t-\log \left(
1+t\right) \geq t-\left( t-t^{2}/2+t^{3}/3\right) $ and so $t^{2}/2\leq
z+t^{3}/3\leq z+2\sqrt{6}z^{3/2}$. If $z\geq 1-\log 2$ then $t\geq 1$ and $%
\log \left( 1+t\right) \leq t\log (2)$ so $z=t-\log \left( 1+t\right) \geq
\left( 1-\log (2)\right) t$ and $t\leq \left( 1-\log (2)\right) ^{-1}z$. But
then $t=z+\log \left( 1+t\right) \leq z+\log \left( 1+\left( 1-\log
(2)\right) ^{-1}z\right) $.
\end{proof}

\begin{lemma}
\label{concentration of order statistics}Let $\left( \gamma _{i}\right)
_{1}^{n}$ be an i.i.d. sample from $\left( 0,1\right) $ with corresponding
order statistics $\left( \gamma _{(i)}\right) _{1}^{n}$ and let $t>0$. With
probability at least $1-3^{-1}\pi ^{2}\exp \left( -t^{2}/2\right) $, the
following event occurs: for all $1\leq k\leq n$, $\gamma _{(k)}$ is bounded
above by both of the following quantities%
\begin{eqnarray}
&&\frac{k}{n+1}\left( 1+\xi _{2}^{-1}\left( \exp \left( \frac{-t^{2}-4\log k%
}{2k}\right) \right) \right)  \label{top order} \\
&&1-\frac{n-k+1}{n+1}\left( 1-\xi _{1}^{-1}\left( \exp \left( \frac{%
-t^{2}-4\log \left( n-k+1\right) }{2(n-k+1)}\right) \right) \right)
\label{bottom order}
\end{eqnarray}%
and with probability at least $1-C\exp \left( -t^{2}/2\right) $ the
following event occurs: for all $1\leq k\leq n$,%
\begin{eqnarray}
\gamma _{(k)} &\leq &1-\frac{n-k}{n}\exp \left( -c\max \left\{ \frac{\left(
t+\sqrt{\log k}\right) \sqrt{k}}{\sqrt{n\left( n-k+1\right) }},\frac{%
t^{2}+\log k}{n-k+1}\right\} \right)  \label{Renyi est} \\
&\leq &\frac{k}{n}+c\frac{n-k}{n}\max \left\{ \frac{\left( t+\sqrt{\log k}%
\right) \sqrt{k}}{\sqrt{n\left( n-k+1\right) }},\frac{t^{2}+\log k}{n-k+1}%
\right\}  \notag  \label{crude Renyi est}
\end{eqnarray}%
\end{lemma}

\begin{remark}
For $k\leq n/2$, (\ref{top order}) gives a typical deviation about the mean
at most $C\sqrt{k\log k}/n$ but breaks down as $t\rightarrow \infty $ and $%
n,k$ are fixed. For $k\geq n/2$ (\ref{bottom order}) gives a typical
deviation at most $C\sqrt{\left( n-k+1\right) \log \left( n-k+1\right) }/n$,
and remains non-trivial (i.e.\thinspace $<1$) for all $1\leq k\leq n$ as $%
t\rightarrow \infty $. For $k\leq n/2$ (\ref{Renyi est}) also gives a
typical deviation of $C\sqrt{k\log k}/n$ : it is not quite as precise as (%
\ref{top order}) (which includes the exact function $\xi _{2}$) for $%
0<t<t_{n,k}$ but eventually improves upon (\ref{top order}) and remains
non-trivial as $t\rightarrow \infty $.
\end{remark}

\begin{proof}[Proof of Lemma \ref{concentration of order statistics}]
If $B$ has a binomial distribution with parameters $(n,p)$, and $np\leq s<n$%
, then using the exponential moment method,%
\begin{equation*}
\mathbb{P}\left\{ B\geq s\right\} =\mathbb{P}\left\{ e^{\lambda B}\geq
e^{\lambda s}\right\} \leq e^{-\lambda s}\left( 1-p+pe^{\lambda }\right)
^{n}=\left( \frac{np}{s}\right) ^{s}\left( \frac{n-np}{n-s}\right) ^{n-s}
\end{equation*}%
See e.g. \cite[Ex. 2.11 p48]{BLMass}. Let $\#\left( E\right) $ denote the
number of $1\leq i\leq n$ such that $\gamma _{i}\in E$. Then (recycling the
variable $s$),%
\begin{eqnarray*}
&\mathbb{P}&\left\{ \gamma _{(k)}\geq \frac{k+s\sqrt{k}}{n+1}\right\}=%
\mathbb{P}\left\{ \#\left( \frac{k+s\sqrt{k}}{n+1},1\right) \geq
n-k+1\right\} \\
&\leq &\left( 1-\frac{s\sqrt{k}}{n-k+1}\right) ^{n-k+1}\left( 1+\frac{s}{%
\sqrt{k}}\right) ^{k-1}\left( \frac{k}{k-1}\right) ^{k-1}\div \left( \frac{%
n+1}{n}\right) ^{n} \\
&\leq &\left( \xi _{2}\left( \frac{s}{\sqrt{k}}\right) \right) ^{k}\leq 
\frac{\exp \left( -t^{2}/2\right) }{k^{2}}
\end{eqnarray*}%
provided%
\begin{equation*}
s\geq \sqrt{k}\xi _{2}^{-1}\left( k^{-2/k}\exp \left( \frac{-t^{2}}{2k}%
\right) \right)
\end{equation*}%
We then apply the union bound over all $1\leq k\leq n$. (\ref{bottom order})
follows the same lines:%
\begin{equation*}
\mathbb{P}\left\{ \gamma _{(k)}\geq \frac{k+s\sqrt{n-k+1}}{n+1}\right\} =%
\mathbb{P}\left\{ \#\left( \frac{k+s\sqrt{n-k+1}}{n+1},1\right) \geq
n-k+1\right\}
\end{equation*}%
To prove (\ref{Renyi est}), we make use of the R\'{e}nyi representation of
order statistics from the exponential distribution (which we heard of from 
\cite[Theorem 2.5]{BouThom}): there exist i.i.d. standard exponential random
variables $\left( Z_{j}\right) _{1}^{n}$ such that%
\begin{equation*}
-\log \left( 1-\gamma _{(k)}\right) =\sum_{j=1}^{k}\frac{Z_{j}}{n-j+1}
\end{equation*}%
(this is an easy consequence of the fact that for all $1\leq k\leq n$, the
order statistics $\left( \gamma _{(j)}\right) _{k+1}^{n}$ are (after being
re-scaled to fill $\left( 0,1\right) $) independent of $\left( \gamma
_{(j)}\right) _{1}^{k}$ and distributed as the order statistics from a
sample of size $n-k$. Thus we may write%
\begin{equation*}
1-\gamma _{(k)}=\left( 1-\gamma _{(1)}\right) \prod_{j=2}^{k}\left( 1-\gamma
_{(j)}\right) \left( 1-\gamma _{(j-1)}\right) ^{-1}
\end{equation*}%
which is the product of $k$ independent variables). Concentration of $\log
\left( 1-\gamma _{(k)}\right) ^{-1}$ about its mean (with probability $%
1-Ck^{-2}\exp \left( -t^{2}/2\right) $) can now be studied using the basic
estimate
\begin{equation}
\mathbb{P}\left\{\left\vert \sum_{j=1}^ka_j(Z_j-1) \right\vert>r\right\}\leq 2\exp\left(-c\min\left\{\left(\frac{r}{\left\vert a \right\vert}\right)^2,\frac{r}{\left\vert a \right\vert_\infty}\right\}\right)\label{sum of exp dev bound}
\end{equation}
valid for all $r>0$ and all $a\in\mathbb{R}^k$. (\ref{sum of exp dev bound}) is proved using the exponential moment method, see for example \cite[Ex. 2.27 p50]{BLMass}, or use \cite{Fr20}[Theorem 3]. The result can be
transferred back to $\gamma _{(k)}$ using the transformation $t\mapsto
1-\exp \left( -t\right) $.
\end{proof}

Recall the definition of the quantile function as a generalized inverse given above the statement of Theorem \ref{poly tails Lip}.

\begin{corollary}
\label{robust simplified bound for sum}Let $n\in \mathbb{N}$, $\lambda \in \left[ 2,\infty \right) $,
and let $(Y_{i})_{1}^{n}$
be an i.i.d. sequence of non-negative random variables, each with cumulative
distribution $F$, quantile function $F^{-1}$, and corresponding order statistics $\left( Y_{(i)}\right)
_{1}^{n}$. With probability at least $1-3^{-1}\pi ^{2}\exp \left( -\lambda
^{2}/2\right) $, the following event holds: for all $j,k\in\mathbb{Z}$ with $0\leq j\leq k<n$,
\begin{eqnarray*}
&&\sum_{i=n-k}^{n-j}Y_{(i)}\leq F^{-1}\left(1-\frac{j+1}{n+1}\left( 1-\xi _{1}^{-1}\left( \exp \left( \frac{%
-\lambda^{2}-4\log \left( j+1\right) }{2(j+1)}\right) \right) \right)\right)\\
&+&\left( n+1\right) \int_{(j+1)/\left( n+1\right) }^{(k+1)/\left( n+1\right) }F^{-1}\left(
1-t\left( 1-\xi _{1}^{-1}\left( \exp \left( \frac{-\lambda ^{2}-4\log \left(
(n+1)t\right) }{2\left( n+1\right) t}\right) \right) \right) \right) dt
\end{eqnarray*}%
\end{corollary}

\begin{proof}
Let  $\left( \gamma _{i}\right) _{1}^{n}$ be an i.i.d. sample from the uniform distribution on $(0,1)$. Since $\left(Y_{(i)}\right)_1^n$ has the same distribution as $\left(F^{-1}\left(
\gamma _{(i)}\right)\right)_1^n$ we may assume without loss of generality that $Y_{(i)}=F^{-1}\left(
\gamma _{(i)}\right) $. We now apply Lemma \ref{concentration of order statistics} to the random vector $\left( \gamma _{(i)}\right)
_{1}^{n}$. If $j=k$ we simply have one term. If $j<k$ write
\[
\sum_{i=n-k}^{n-j}Y_{(i)}=Y_{(n-j)}+\sum_{i=n-k}^{n-j-1}Y_{(i)}
\]
and compare the sum to an integral using right hand endpoints the fact that the integrand is decreasing. Here we also use the fact that $x\mapsto \left( \lambda ^{2}+4\log x\right) /x$ is decreasing provided $\log
x\geq 1-\lambda ^{2}/4$, and we have assumed that $\lambda \geq 2$.
\end{proof}

\begin{lemma}\label{replacio}
In Corollary \ref{robust simplified bound for sum}, we can replace the upper bound for $\sum_{i=n-k}^{n-j}Y_{(i)}$ with
\begin{eqnarray*}
&&F^{-1}\left(
1-\frac{j+1}{n+1}e^{-1}\exp \left( \frac{-\lambda ^{2}-4\log \left(
j+1\right) }{2\left(j+1\right)}\right) \right)\\
&+&\lambda^2\int_{\frac{2(j+1)}{\lambda^2}\exp\left(\frac{-\lambda^2}{2(j+1)}\right)}^{\frac{2(k+1)}{\lambda^2}\exp\left(\frac{-\lambda^2}{2(k+1)}\right)}F^{-1}\left(1-e^{-1-2/e}\frac{\lambda^2}{2(n+1)}z\right)\left\{1+\frac{1}{z}\left(\log\left(e+\frac{1}{z}\right)\right)^{-2}\right\}dz
\end{eqnarray*}
\end{lemma}

\begin{proof}
By Corollary \ref{robust simplified bound for sum} and Lemma \ref{numebecali}, $\sum_{i=n-k}^{n-j}Y_{(i)}$ is bounded above by
\begin{eqnarray*}
&&F^{-1}\left(
1-\frac{j+1}{n+1}e^{-1}\exp \left( \frac{-\lambda ^{2}-4\log \left(
j+1\right) }{2\left(j+1\right)}\right) \right)\\
&+&(n+1)\int_{(j+1)/\left( n+1\right) }^{(k+1)/\left( n+1\right) }F^{-1}\left(
1-e^{-1-2/e}t\exp \left( \frac{-\lambda ^{2}}{2\left( n+1\right) t}\right)
\right) dt
\end{eqnarray*}
Then set
\[
s=\frac{\lambda^2}{2(n+1)}t^{-1}
\]
and the integral becomes
\begin{equation}
\frac{\lambda^2}{2}\int_{\frac{\lambda^2}{2(k+1)}}^{\frac{\lambda^2}{2(j+1)}}F^{-1}\left(1-e^{-1-2/e}\frac{\lambda^2}{2(n+1)}s^{-1}e^{-s}\right)s^{-2}ds\label{int quantii}
\end{equation}
Setting $z=q(s)=s^{-1}e^{-s}$ and using $q'(s)=-q(s)\left(1+1/s\right)$,
\[
ds=\frac{dz}{q'\left(q^{-1}\left(z\right)\right)}=\frac{dz}{q'(s)}=\frac{dz}{-q(s)(1+s^{-1})}=\frac{-e^sdz}{s^{-1}(1+s^{-1})}
\]
The expression in (\ref{int quantii}) can then be written as
\begin{equation}
\frac{\lambda^2}{2}\int_{\frac{2(j+1)}{\lambda^2}\exp\left(\frac{-\lambda^2}{2(j+1)}\right)}^{\frac{2(k+1)}{\lambda^2}\exp\left(\frac{-\lambda^2}{2(k+1)}\right)}F^{-1}\left(1-e^{-1-2/e}\frac{\lambda^2}{2(n+1)}z\right)\frac{e^s}{1+s}dz
\end{equation}
We'd like to write $e^s/(1+s)$ as a function of $x$, or at least bound it above by such a function, and we start by estimating it in terms of $z=s^{-1}e^{-s}\in(0,\infty)$. When $z$ is small $s$ is large and $s\geq (1/2)\log z^{-1}$, and
\[
\frac{e^s}{1+s}\leq C\frac{1}{z}\left(\log\frac{1}{z}\right)^{-2}
\]
When $z$ is large $s$ is small and $e^s/(1+s)\leq C$. By continuity, for all $z\in(0,\infty)$,
\[
\frac{e^s}{1+s}\leq C\left\{1+\frac{1}{z}\left(\log\left(e+\frac{1}{z}\right)\right)^{-2}\right\}
\]
and we define
\[
C_0=\sup\left\{\frac{e^s}{1+s}\left[1+\frac{se^s}{\left(\log\left(e+se^s\right)\right)^{2}}\right]^{-1}:s\in(0,\infty)\right\}
\]
to be the smallest possible value of $C$. A numerical computation shows that $1<C_0<2$.
\end{proof}

\begin{lemma}\label{lemlem smalio}
For all $a,b\in(0,\infty)$ with $a\leq b$ and all $r\in\mathbb{R}$,
\begin{equation}
\int_a^bx^{-r}dx\leq C\min\left\{\left\vert 1-r\right\vert^{-1},\log \frac{b}{a}\right\}\left(a^{1-r}+b^{1-r}\right)\label{power interp}
\end{equation}
where we define $0^{-1}=\infty$. If $0<a<b<e^{-1}$ and $r>1$ then
\begin{eqnarray}
&&\int_a^bx^{-r}\left(\log\frac{1}{x}\right)^{-2}dx\leq C\min\left\{1,\log\frac{\log\frac{1}{a}}{\log\frac{1}{b}}\right\}\left(\log\frac{1}{b}\right)^{-1}\nonumber\\
&+&C\min\left\{(r-1)^{-1},\log \frac{b}{a}\right\}\left[(r-1)^{-1}+\log\frac{1}{a}\right]^{-2}a^{1-r}\label{power log intii}\\
&\leq& C_r\min\left\{1,\log \frac{b}{a}\right\}\left(\log\frac{1}{a}\right)^{-2}a^{1-r}\label{const r}
\end{eqnarray}
The inequalities in (\ref{power interp}) and (\ref{power log intii}) can be reversed by replacing $C$ with $c$, and (\ref{const r}) can be reversed by replacing $C_r$ with $c_r$.
\end{lemma}
\begin{proof}
Assume without loss of generality that $a<b$. (\ref{power interp}) is Lemma 3 in \cite{Fr22} without the restriction that $a=1$, and it follows from that lemma by a change of variables. For (\ref{power log intii}), set $t=(r-1)\log(1/x)$, so the integral becomes
\begin{eqnarray}
(r-1)\int_{(r-1)\log(1/b)}^{(r-1)\log(1/a)}t^{-2}e^{t}dt
\end{eqnarray}
Now $t^{-2}e^{t}$ is the same order of magnitude as $t^{-2}+(1+t)^{-2}e^{1+t}$, which can be checked seperately for $t\leq 1$ and $t>1$. To integrate the first term of this integrand use (\ref{power interp}) with $2$ in place of $r$. To integrate the second term set $u=t+1$ and note that the resulting integrand is the same order of magnitude as a function that has an instantaneous exponential growth rate that is bounded above and below by universal constants, i.e.
\[
C(r-1)\int_{(r-1)\log(1/b)+1}^{(r-1)\log(1/a)+1}u^{-2}e^udu\\
\leq C(r-1)\int_{(r-1)\log(1/b)+1}^{(r-1)\log(1/a)+1}(u+2)^{-2}e^udu
\]
and for all $u\in[1,\infty)$,
\[
\frac{2}{3}\leq\frac{d}{du}\ln\left((u+2)^{-2}e^u\right)\leq 1
\]
so the integral is the same order of magnitude as
\begin{eqnarray*}
C(r-1)\min\left\{1,(r-1)\log \frac{b}{a}\right\}\left[(r-1)\log\frac{1}{a}+3\right]^{-2}\exp\left((r-1)\log(1/a)+1\right)
\end{eqnarray*}
\end{proof}

\begin{proposition}\label{final productiones}
Consider the setting and assumptions of Corollary \ref{robust simplified bound for sum} and assume, in addition, that $p>0$, $T\geq 1$ and that for all $\delta,x\in(0,1)$,
\begin{equation}
H^*(\delta x)\geq T^{-1} \delta^{-1/p}H^*(x)\label{condit fast grw}
\end{equation}
where $H^*(x)=F^{-1}(1-x)$. Then the upper bound for $\sum_{i=n-k}^{n-j}Y_{(i)}$ can be replaced with
\begin{eqnarray*}
\left[1+T\lambda^2A\right]H^*\left(e^{-1-2/e}\frac{j+1}{n+1}\exp\left(\frac{-\lambda^2}{2(j+1)}\right)\right)+Cn\int_{\frac{j+1}{n+1}\exp\left(\frac{-\lambda^2}{2(j+1)}-1-2/e\right)}^{\frac{k+1}{n+1}\exp\left(\frac{-\lambda^2}{2(k+1)}-1-2/e\right)}H^*(x)dx
\end{eqnarray*}
where $A=0$ if $\lambda^2/2\leq j+1$ and $A$ equals
\begin{eqnarray*}
&&C^{1+1/p}\min\left\{p,\lambda^2\left(\frac{1}{j+1}-\frac{1}{\min\left\{\lambda^2/2,k+1\right\}}\right)\right\}\left(p+1+\frac{\lambda^2}{j+1}\right)^{-2}\\
&&+C\min\left\{1,\log\frac{\min\left\{k+1,\lambda^2/2\right\}}{j+1}\right\}\left[\frac{\lambda^2}{2(j+1)}\exp\left(\frac{\lambda^2}{2(j+1)}\right)\right]^{-1/p}\left[1+\frac{\lambda^2}{k+1}\right]^{-1}
\end{eqnarray*}
if $\lambda^2/2>j+1$.
\end{proposition}

\begin{proof}
By Lemma \ref{replacio} and assumption (\ref{condit fast grw}), $\sum_{i=n-k}^{n-j}Y_{(i)}$ is (with the required probability) at most $I+II+III$, where
\begin{eqnarray*}
I&=&H^*\left(
\frac{j+1}{n+1}e^{-1}\exp \left( \frac{-\lambda ^{2}-4\log \left(
j+1\right) }{2\left(j+1\right)}\right) \right)\\
II&=&C\lambda^2\int_{\frac{2(j+1)}{\lambda^2}\exp\left(\frac{-\lambda^2}{2(j+1)}\right)}^{\frac{2(k+1)}{\lambda^2}\exp\left(\frac{-\lambda^2}{2(k+1)}\right)}H^*\left(e^{-1-2/e}\frac{\lambda^2}{2(n+1)}z\right)dz\\
&\leq&Cn\int_{\frac{j+1}{n+1}\exp\left(\frac{-\lambda^2}{2(j+1)}-1-2/e\right)}^{\frac{k+1}{n+1}\exp\left(\frac{-\lambda^2}{2(k+1)}-1-2/e\right)}H^*\left(x\right)dx
\end{eqnarray*}
and $III$ is the product of
\begin{eqnarray}
C\lambda^2TH^*\left(e^{-1-2/e}\frac{j+1}{n+1}\exp\left(\frac{-\lambda^2}{2(j+1)}\right)\right)\label{firs threi}
\end{eqnarray}
and
\begin{eqnarray*}
\int_{\min\left\{e^{-1},\frac{2(j+1)}{\lambda^2}\exp\left(\frac{-\lambda^2}{2(j+1)}\right)\right\}}^{\min\left\{e^{-1},\frac{2(k+1)}{\lambda^2}\exp\left(\frac{-\lambda^2}{2(k+1)}\right)\right\}}\left(\frac{2(j+1)}{\lambda^2}\exp\left(\frac{-\lambda^2}{2(j+1)}\right)z^{-1}\right)^{1/p}\left\{\frac{1}{z}\left(\log\left(e+\frac{1}{z}\right)\right)^{-2}\right\}dz
\end{eqnarray*}
Here what we are doing is taking the term $1+\frac{1}{z}\left(\log\left(e+\frac{1}{z}\right)\right)^{-2}$ which appears in Lemma \ref{replacio} and expressing the corresponding integral as a sum of terms, one with coefficient $1$ and another with coefficient $\frac{1}{z}\left(\log\left(e+\frac{1}{z}\right)\right)^{-2}$. The second term only comes into play when the coefficient is at least $c$, so we may restrict the integral in $III$ to values of $z$ in $(0,e^{-1})$. By Lemma \ref{lemlem smalio} we can bound the integral in $III$ above by the sum of
\begin{eqnarray*}
&&C^{1+1/p}\min\left\{p,\log\frac{\max\left\{e,\frac{\lambda^2}{2(j+1)}\exp\left(\frac{\lambda^2}{2(j+1)}\right)\right\}}{\max\left\{e,\frac{\lambda^2}{2(k+1)}\exp\left(\frac{\lambda^2}{2(k+1)}\right)\right\}}\right\}\left(p+1+\frac{\lambda^2}{2(j+1)}\right)^{-2}\\
&&\times\left[1+\frac{2(j+1)}{\lambda^2}\exp\left(\frac{-\lambda^2}{2(j+1)}\right)\right]^{1/p}
\end{eqnarray*}
and
\begin{eqnarray*}
&&C\left[\frac{\lambda^2}{2(j+1)}\exp\left(\frac{\lambda^2}{2(j+1)}\right)\right]^{-1/p}\left[\max\left\{1,\log\left[\frac{\lambda^2}{2(k+1)}\exp\left(\frac{\lambda^2}{2(k+1)}\right)\right]\right\}\right]^{-1}\\
&&\times\min\left\{1,\log\frac{\max\left\{1,\log\left[\frac{\lambda^2}{2(j+1)}\exp\left(\frac{\lambda^2}{2(j+1)}\right)\right]\right\}}{\max\left\{1,\log\left[\frac{\lambda^2}{2(k+1)}\exp\left(\frac{\lambda^2}{2(k+1)}\right)\right]\right\}}\right\}
\end{eqnarray*}
Unless $\lambda^2/(2(j+1))\geq 1$, the integral in $III$ is zero because the interval of integration has length zero. So while bounding $III$ we assume this is the case, and this allows for simplification. By considering the cases $\lambda^2\leq 2(k+1)$ and $\lambda^2> 2(k+1)$ separately,
\begin{eqnarray*}
&&\min\left\{p,\log\frac{\max\left\{e,\frac{\lambda^2}{2(j+1)}\exp\left(\frac{\lambda^2}{2(j+1)}\right)\right\}}{\max\left\{e,\frac{\lambda^2}{2(k+1)}\exp\left(\frac{\lambda^2}{2(k+1)}\right)\right\}}\right\}\\
&\leq& C\min\left\{p,\lambda^2\left(\frac{1}{j+1}-\frac{1}{\min\left\{\lambda^2/2,k+1\right\}}\right)+\log\frac{\min\left\{\lambda^2/2,k+1\right\}}{j+1}\right\}\\
&\leq& C\min\left\{p,\lambda^2\left(\frac{1}{j+1}-\frac{1}{\min\left\{\lambda^2/2,k+1\right\}}\right)\right\}
\end{eqnarray*}
Here we have used the fact that because the logarithm is $1$-Lipschitz on $[1,\infty)$, for all $a,b\in[1,\infty)$ with $a\leq b$, $\log(1/a)-\log(1/b)\leq 1/a-1/b$. We now prove two claims which help simplify another term.

\underline{Claim:} for all $a,b\in[1,\infty)$ such that $a<b$ and $a\leq e+1$ (say),
\[
c\min\left\{\frac{b-a}{a},\log b\right\}\leq\log b -\log a \leq \min\left\{\frac{b-a}{a},\log b\right\}
\]
\underline{Proof of Claim:} The upper bound holds because the derivative of $\log$ is decreasing and because $\log a\geq 0$. For the lower bound, note that $a^2\leq (e+1)a$, so either $b\leq (e+1)a$ or $b>a^2$. If $b\leq (e+1)a$ then (by considering the derivative)
\[
\log b -\log a\geq\frac{b-a}{b}\geq C\frac{b-a}{a}
\]
and if $b>a^2$ then
\[
\log b -\log a=(1/2)\log b
\]
\underline{Claim:} for all $s,t\in[1,\infty)$ with $s<t$,
\[
\log\frac{t+\log t}{s+\log s}\leq C \log\frac{t}{s}
\]
\underline{Proof of Claim:} This is certainly true when $s\geq e$, because $x\mapsto (\log x)/x$ is decreasing on $(e,\infty)$, which implies the desired inequality with $C=1$. For $s<e$, apply the first claim twice to get
\[
\log\frac{t+\log t}{s+\log s}\leq C \min\left\{\frac{(t-s)+\log t-\log s}{s+\log s},\log(t+\log t)\right\}\leq C \min\left\{\frac{t-s}{s},\log(t)\right\}
\]
which proves the second claim.

We now find a simplified upper bound for
\begin{eqnarray*}
\min\left\{1,\log\frac{\max\left\{1,\log\left[\frac{\lambda^2}{2(j+1)}\exp\left(\frac{\lambda^2}{2(j+1)}\right)\right]\right\}}{\max\left\{1,\log\left[\frac{\lambda^2}{2(k+1)}\exp\left(\frac{\lambda^2}{2(k+1)}\right)\right]\right\}}\right\}
\end{eqnarray*}
When $\lambda^2\geq 2(k+1)$ and when $\lambda^2< 2(k+1)$ we get (respectively) as upper bounds using the second claim and using $\log\log (te^t)\leq C\log t$ for $t\geq 1$,
\[
\min\left\{1,\log\frac{k+1}{j+1}\right\}\hspace{2cm}C\min\left\{1,\log\frac{\lambda^2}{2(j+1)}\right\}
\]
In either case, we have the following upper bound
\[
C\min\left\{1,\log\frac{\min\left\{k+1,\lambda^2/2\right\}}{j+1}\right\}=C\log\frac{\min\left\{e(j+1),k+1,\lambda^2/2\right\}}{j+1}
\]
The result of these simplifications is that the upper bound for the integral in $III$ reduces to the quantity $A$ as defined in the statement of the result.
\end{proof}

\begin{remark}\label{remarkiou}
If in Proposition \ref{final productiones} we set $j=0$, $k=n-1$ and assume that $p>1$ and $T\leq C$ for any desired constant $C\geq 1$, the bound on $\sum_{i=1}^{n}Y_{(i)}=\sum_{i=1}^{n}Y_{i}$ can be replaced with
\begin{eqnarray*}
C\left(1+\lambda^{-2/p}e^{-\lambda^2/(2p)}\min\{\lambda^2,n\}\right)H^*\left(e^{-1-2/e}\frac{1}{n+1}\exp\left(\frac{-\lambda^2}{2}\right)\right)+Cn\mathbb{E}Y_1
\end{eqnarray*}
One also has
\[
1+\lambda^{-2/p}e^{-\lambda^2/(2p)}\min\{\lambda^2,n\}\leq Cp
\]
\end{remark}

\begin{proof}
When bounding $A$ we may assume without loss of generality that $A\neq0$. Distribute $\lambda^2$ into $A$, bound the two minima in $A$ by $p$ and $1$ respectively and use
\[
\frac{\lambda^2p}{\left(p+1+\lambda^2\right)^2}\leq 1 \hspace{1cm} \text{and}\hspace{1cm} \lambda^{2-2/p}\exp\left(-\frac{\lambda^2}{2p}\right)\leq Cp
\]
To bound the coefficient as in the statement, use $\min\{\lambda^2,n\}\leq \lambda^2$ and optimize.
\end{proof}

\begin{corollary}\label{eg sum poly tail}
Let $p>1$, $n\in \mathbb{N}$, $\lambda>0$, and let $(Y_{i})_{1}^{n}$
be an i.i.d. sequence of non-negative random variables, each with cumulative
distribution $F(x)=\min\{1,x^{-p}\}$, quantile function $F^{-1}(x)=x^{-1/p}$, and corresponding order statistics $\left( Y_{(i)}\right)
_{1}^{n}$. With probability at least $1-C\exp \left( -\lambda
^{2}/2\right)$, the following event occurs: for all $j\in\mathbb{Z}$ with $0\leq j\leq n/2$, $\sum_{i=1}^{n-j}Y_{(i)}$ is bounded above by
\[
\frac{Cpn}{p-1}+C\left(1+(j+1)\min\left\{\left(\frac{\lambda^2}{p(j+1)}\right)^2,\left(\frac{\lambda^2}{p(j+1)}\right)^{-1}\right\}\right)\left(\frac{n}{j+1}\right)^{1/p}\exp\left(\frac{\lambda^2}{2p(j+1)}\right)
\]
\end{corollary}

\noindent \textbf{Note:} The condition $j\leq n/2$ is not necessary, but rather highlights the setting where the bound is most effective.

\begin{proof}
Proposition \ref{final productiones} with $k=n-1$ gives the estimate $\sum_{i=1}^{n-j}Y_{(i)}\leq I+II+III+IV$ where
\begin{eqnarray*}
I&=&Cn\int_0^1H^*(x)dx=\frac{Cpn}{p-1}\\
II&=&H^{*}\left(e^{-1-2/e}\frac{j+1}{n+1}\exp\left(\frac{\lambda^2}{2(j+1)}\right)\right)\leq C\left(\frac{n}{j+1}\right)^{1/p}\exp\left(\frac{\lambda^2}{2p(j+1)}\right)
\end{eqnarray*}
and $III,IV=0$ unless $\lambda^2\geq 2(j+1)$, in which case
\begin{eqnarray*}
III&=&C\lambda^2\min\left\{p,\frac{\lambda^2}{j+1}\right\}\max\left\{p,\frac{\lambda^2}{j+1}\right\}^{-2}H^{*}\left(e^{-1-2/e}\frac{j+1}{n+1}\exp\left(\frac{\lambda^2}{2(j+1)}\right)\right)\\
&\leq&C\lambda^2\min\left\{p,\frac{\lambda^2}{j+1}\right\}\max\left\{p,\frac{\lambda^2}{j+1}\right\}^{-2}\left(\frac{n}{j+1}\right)^{1/p}\exp\left(\frac{\lambda^2}{2p(j+1)}\right)
\end{eqnarray*}
and $IV$ equals
\begin{eqnarray*}
&&C\lambda^2\left[\frac{\lambda^2}{2(j+1)}\exp\left(\frac{\lambda^2}{2(j+1)}\right)\right]^{-1/p}\left(1+\frac{\lambda^2}{n}\right)^{-1}H^{*}\left(e^{-1-2/e}\frac{j+1}{n+1}\exp\left(\frac{\lambda^2}{2(j+1)}\right)\right)\\
&&\leq C\lambda^{2-2/p}n^{1/p}\left(1+\frac{\lambda^2}{n}\right)^{-1}
\end{eqnarray*}
for $\lambda\leq \sqrt{n}$, $IV\leq C\lambda^{2-2/p}n^{1/p}\leq Cn$ and for $\lambda> \sqrt{n}$, $IV\leq C\lambda^{-2/p}n^{1-1/p}<Cn$. Either way, $IV\leq I$.
\end{proof}

\bigskip

\underline{Further remarks under the tail condition $F(x)=\min\{1,x^{-p}\}$.}

\bigskip

Setting $j=0$ in Corollary \ref{eg sum poly tail}, or applying Remark \ref{remarkiou},
\begin{equation*}
\mathbb{P}\left\{ \sum_{i=1}^{n}Y_{i}\geq \frac{Cpn}{p-1}+Cn^{1/p}\exp \left( 
\frac{\lambda ^{2}}{2p}\right)  \right\}
<Ce^{-\lambda^2/2}
\end{equation*}%
This gives the correct order of magnitude for $\sum_{1}^{n}Y_{i}$ in the
i.i.d. case up to the value of $C$, since the
same bound describes the order of magnitude of $n\mathbb{E}Y_{1}+\max_{1\leq
i\leq n}Y_{i}$.

Returning to the case of a general value of $j$, and setting $k=j+1$ and $s^{k}=\exp \left( \lambda ^{2}/2\right) $, the bound in Corollary \ref{eg sum poly tail} can be written as
\begin{equation*}
\mathbb{P}\left\{ \sum_{i=1}^{n-k+1}Y_{(i)}>\frac{Cpn}{p-1}%
+C\left( 1+k\min\left\{\left(\frac{\log s}{p}\right)^2,\left(\frac{\log s}{p}\right)^{-1}\right\}\right)\left(\frac{n}{k}\right)^{1/p}s^{1/p}\right\}\leq Cs^{-k}
\end{equation*}%
for $s>1$. Compare this to the following bound of Gu\'{e}don, Litvak, Pajor, and Tomczak-Jaegermann \cite[Lemma 4.4]{GLPTJ17}: for all $s\in
\left( 1,\infty \right) $,%
\begin{equation*}
\mathbb{P}\left\{ \sum_{i=1}^{n-k+1}Y_{(i)}> \frac{12p\left( es\right)
^{1/p}}{p-1}n\right\} \leq s^{-k}
\end{equation*}%

\subsection{Partial reduction to the case of equal coefficients (geometric approach)}

\subsubsection{\label{dual norm}A norm for quantiles of linear functionals}

Let $\mu $ be any probability measure on $\mathbb{R}^{n}$ not supported on
any half space not containing the origin, and such that%
\begin{equation}
\int_{\mathbb{R}^{n}}\left\vert \left\langle x,a\right\rangle \right\vert
d\mu (x)<\infty  \label{first moment assump}
\end{equation}%
for all $a\in \mathbb{R}^{n}$, let $X=\left( X_{i}\right) _{1}^{n}$ be a
random vector with distribution $\mu $, and let $F_{a}(t)=%
\mathbb{P}\left\{ \sum_1^n a_{i}X_{i}\leq t\right\} $. Set $X^{(0)}=0\in \mathbb{R}^{n}$ and let $%
\left( X^{(j)}\right) _{1}^{\infty }$ be an i.i.d. sample from $\mu $, let $%
\delta \in \left( 0,1/2\right) $, and let $N\sim Pois(\delta ^{-1})$. A
basic result in the theory of Poisson point processes is that the random
measure%
\begin{equation*}
\sum_{j=1}^{N}\delta \left( X^{(j)}\right)
\end{equation*}%
is a Poisson point process with intensity $\delta ^{-1}\mu $, where $\delta
\left( x\right) $ denotes the Dirac point mass at $x$, not to be confused
with $\delta \in \left( 0,1/2\right) $. The set%
\begin{equation*}
\mathfrak{Z}=\mathbb{E}conv\left\{ X_{i}\right\} _{0}^{N}:=\left\{ x\in 
\mathbb{R}^{n}:\forall \theta \in S^{n-1},\left\langle \theta
,x\right\rangle \leq \mathbb{E}\max_{0\leq j\leq N}\left\langle \theta
,X^{(j)}\right\rangle \right\}
\end{equation*}%
is seen to be a compact convex set with nonempty interior (i.e. a convex
body), in fact $0\in int\left( \mathfrak{Z}\right) $. Its dual Minkowski
functional, given by $\left\vert a\right\vert _{\mathfrak{Z}^{\circ }}=\sup
\left\{ \left\langle x,a\right\rangle :x\in \mathfrak{Z}\right\} $, can be
expressed as%
\begin{equation}
\left\vert a\right\vert _{\mathfrak{Z}^{\circ }}=\mathbb{E}\max_{0\leq j\leq
N}\left\langle a,X^{(j)}\right\rangle=\delta ^{-1}\int_{0}^{1-F_{a}\left( 0\right) }F_{a}^{-1}\left( 1-s\right)
\exp \left( -\delta ^{-1}s\right) ds\label{Mink fnc Poisson}
\end{equation}%
This is because for $t>0$, by definition of a Poisson point process,%
\begin{equation*}
G(t):=\mathbb{P}\left\{ \max_{0\leq j\leq N}\left\langle
a,X^{(j)}\right\rangle \leq t\right\} =\exp \left( -\delta ^{-1}\left(
1-F_{a}(t)\right) \right)
\end{equation*}%
so
\[
\mathbb{E}\max_{0\leq j\leq N}\left\langle a,X^{(j)}\right\rangle
=\int_{0}^{1}G^{-1}(t)dt=\int_{\mathbb{P}\left\{ \max =0\right\}
}^{1}F_{a}^{-1}\left( 1-\delta \log t^{-1}\right) dt
\]

This convex body is a modification of the expected convex hull of a fixed
sample size used in \cite{FrVi} (see references therein) and is related to
the dual (polar) of the convex floating body defined by deleting all half
spaces with $\mu $ measure less than $\delta $, see \cite{Bob10, Fr0, ScWer}. Its
advantage over the convex floating body is that there is an explicit formula
for its Minkowski functional (by definition), and its advantage over the
expected convex hull with a fixed sample size is the representation of its dual Minkowski functional in (\ref{Mink fnc Poisson}).

\begin{lemma}
\label{convexity of quantiles}
For all $a\in \mathbb{R}^{n}$ and all $0<\delta <1-F_{a}\left( 0\right) $,
\begin{equation*}
\mathbb{P}\left\{ \sum_{i=1}^{n}a_{i}X_{i}>2\left\vert a\right\vert _{%
\mathfrak{Z}^{\circ }}\right\} \leq \delta \log 2\hspace{0.65in}\mathbb{P}%
\left\{ \sum_{i=1}^{n}a_{i}X_{i}\geq \left( 1+R\right) ^{-1}\left\vert
a\right\vert _{\mathfrak{Z}^{\circ }}\right\} \geq \delta
\end{equation*}
where
\begin{equation*}
R=\frac{\delta ^{-1}\int_{1-\delta }^{1}F_{a}^{-1}\left( t\right) dt}{%
F_{a}^{-1}\left( 1-\delta \right) }
\end{equation*}%

\end{lemma}

\begin{proof}
Comparing the mean $\mathbb{E}$ and any median $\mathbb{M}$,%
\begin{eqnarray*}
\mathbb{P}\left\{ \sum_{i=1}^{n}a_{i}X_{i}>2\left\vert a\right\vert _{%
\mathfrak{Z}^{\circ }}\right\} &\leq &\mathbb{P}\left\{
\sum_{i=1}^{n}a_{i}X_{i}>\mathbb{M}\max_{0\leq j\leq N}\left\langle
a,X^{(j)}\right\rangle \right\} \\
&=&-\delta \log G\left( \mathbb{M}\max_{0\leq j\leq N}\left\langle
a,X^{(j)}\right\rangle \right) \leq \delta \log 2
\end{eqnarray*}%
On the other hand, from (\ref{Mink fnc Poisson}),
\begin{eqnarray*}
\left\vert a\right\vert _{\mathfrak{Z}^{\circ
}}\leq \delta ^{-1}\int_{0}^{\delta }F_{a}^{-1}\left( 1-s\right) ds+\delta
^{-1}\int_{0}^{\infty }F_{a}^{-1}\left( 1-\delta \right) \exp \left( -\delta
^{-1}s\right) ds
\end{eqnarray*}%
so for all $\varepsilon \in \left( 0,1/2\right) $,%
\begin{equation*}
\mathbb{P}\left\{ \sum_{i=1}^{n}a_{i}X_{i}>\left( 1+\varepsilon \right)
^{-1}\left( 1+R\right) ^{-1}\left\vert a\right\vert _{\mathfrak{Z}^{\circ
}}\right\} \geq \mathbb{P}\left\{ \sum_{i=1}^{n}a_{i}X_{i}>\left(
1+\varepsilon \right) ^{-1}F_{a}^{-1}\left( 1-\delta \right) \right\} >\delta
\end{equation*}%
The role of $\varepsilon $ is a technicality related to the definition of
the generalized inverse $F_{a}^{-1}$.
\end{proof}

If, on the other hand, $\mu$ is supported on $[0,\infty)^n$ and for all $a\in\mathbb{R}^n$
\[
\mathbb{P}\left\{\sum_{i-1}^n\left\vert a_i\right\vert X_i>0\right\}>0
\]
and (\ref{first moment assump}) holds, then
\begin{equation}
\left[a\right]_\delta=\mathbb{E}\max_{0\leq j\leq
N}\sum_{i=1}^n\left\vert a_i\right\vert X_i^{(j)}\label{sqrbrnoerm}
\end{equation}
as a function of $a$, is a norm, and Lemma \ref{convexity of quantiles} holds with $\left\vert \cdot\right\vert _{\mathfrak{Z}^{\circ
}}$ replaced with $\left[\cdot\right]_\delta$ and $F_a$ replaced with $F_{(\left\vert a_i\right\vert)_1^n}$. Assuming for simplicity that each $a_i\geq 0$, the version of (\ref{Mink fnc Poisson}) for $\left[\cdot\right]_\delta$ is
\begin{equation}
\left[a\right]_\delta=\delta ^{-1}\int_{0}^{1}F_{a}^{-1}\left( 1-s\right)
\exp \left( -\delta ^{-1}s\right) ds\label{Mink fnc Poisson uncon}
\end{equation}%

\subsubsection{\label{u b norm}A norm characterized by its values on $\{0,1\}^n$}

For any $r\in \left[
1,\infty \right) $ and $q\in \left( 1,\infty \right) $ define
\begin{equation*}
V_{r,q}=\left\{ \max \left\{ \left\vert u\right\vert _{1},r\left\vert
u\right\vert _{q}\right\} ^{-1}u:u\in \left\{ 0,\pm 1\right\} ^{n},u\neq
0\right\}\hspace{2cm}E_{r,q}=conv\left(V_{r,q}\right)
\end{equation*}%
where $conv$ denotes convex hull. The Minkowski functional of $E_{r,q}$ is the norm $\left\vert x\right\vert
_{r,q}=\inf \left\{ \lambda >0:x\in \lambda E_{r,q}\right\} $.

\begin{lemma}\label{formuii}
For all $x\in \left\{ 0,\pm
1\right\} ^{n}$, $\left\vert x\right\vert _{r,q}=\max \left\{ \left\vert
x\right\vert _{1},r\left\vert x\right\vert _{q}\right\}$.
\end{lemma}
\begin{proof}
Since $V_{r,q}\subset \partial \left(B_{1}^{n}\cap r^{-1}B_{q}^{n}\right)$ and $B_{1}^{n}\cap r^{-1}B_{q}^{n}$ is convex, it follows that $V_{r,q}\subset \partial E_{r,q}$.
\end{proof}

\begin{lemma}\label{comparing norms}
If $\left\Vert \cdot \right\Vert $ is any norm on $%
\mathbb{R}^{n}$ and $\left\Vert x\right\Vert \leq \left\vert x\right\vert
_{r,q}$ for all $x\in \left\{ 0,\pm 1\right\} ^{n}$, then $\left\Vert
x\right\Vert \leq \left\vert x\right\vert _{r,q}$ for all $x\in \mathbb{R}%
^{n}$.
\end{lemma}

\begin{proof}
This follows since
\begin{eqnarray*}
conv\left\{ \left\vert u\right\vert
_{r,q}^{-1}u:u\in \left\{ 0,\pm 1\right\} ^{n},u\neq
0\right\}\subseteq conv\left\{ \left\Vert u\right\Vert^{-1}u:u\in \left\{ 0,\pm 1\right\} ^{n},u\neq
0\right\}
\end{eqnarray*}%
By Lemma \ref{formuii}, $LHS$ is $E_{r,q}$, and $RHS$ is a subset of the unit ball corresponding to $\left\vert \cdot\right\vert
_{r,q}$.

\end{proof}

The dual Minkowski functional of $E_{r,q}$ is defined by
\begin{eqnarray*}
\left\vert y\right\vert _{r,q}^{\circ }=\sup \left\{
\sum_{i=1}^{n}x_{i}y_{i}:x\in E_{r,q}\right\}
\end{eqnarray*}%
Recall that $(y_{[i]})_1^n$ denote the non-increasing rearrangement of the absolute values of $(y_i)_1^n$.

\begin{proposition}\label{estimates normii hieri}
For all $x,y\in\mathbb{R}^n$,
\begin{eqnarray}
\left\vert y\right\vert _{r,q}^{\circ } &\leq &2\sup \left\{ r^{-1}k^{-1/q}\sum_{i=1}^{k}y_{[i]}:1\leq k\leq \min
\left\{ r^{q/(q-1)},n\right\} \right\} \leq 2\left\vert y\right\vert
_{r,q}^{\circ }\label{dual mink lor fnc}
\end{eqnarray}%
and
\begin{equation}
\left\vert x\right\vert _{r,q} \leq 4q^{-1}\left( \left\vert x\right\vert
_{1}+r\sum_{i=1}^{n}i^{-1+1/q}x_{[i]}\right) \leq 16\left\vert
x\right\vert _{r,q}\label{mink lor fnc}
\end{equation}%
\end{proposition}

\begin{proof}
The right hand inequality in (\ref{dual mink lor fnc}) follows from the definition of $\left\vert y\right\vert _{r,q}^{\circ }$, since the supremum is an upper bound. For the left hand inequality, note that
\[
\left\vert y\right\vert _{r,q}^{\circ }=\sup \left\{
\sum_{i=1}^{n}x_{i}y_{i}:x\in V_{r,q}\right\}
\]
Since $V_{r,q}$ is invariant under coordinate permutations and coordinate sign changes, so is $E_{r,q}$, and
\begin{eqnarray*}
\left\vert y\right\vert _{r,q}^{\circ }=\left\vert \left(y_{[i]}\right)_1^n\right\vert _{r,q}^{\circ }=\sup\left\{\max \left\{ k,rk^{1/q}\right\} ^{-1}\sum_{i=1}^ky_{[i]}:1\leq k \leq n\right\}
\end{eqnarray*}
For $k\geq r^{q/(q-1)}$, $\max \left\{ k,rk^{1/q}\right\}=k$ and $k^{-1}\sum_1^ky_{[i]}$ is non-increasing in $k$, so we may restrict our attention to values of $k$ such that $k\leq \lceil r^{q/(q-1)}\rceil$. The factor of $2$ is the price we pay for neglecting $k=\lceil r^{q/(q-1)}\rceil$. For (\ref{mink lor fnc}), assume without loss of generality that the coordinates of $x$ are strictly positive and strictly decreasing. Since the canonical embedding of a normed space into its bidual is an
isometry,
\begin{equation}
\left\vert x\right\vert _{r,q}=\sup \left\{
\sum_{i=1}^{n}x_{i}y_{i}:\left\vert y\right\vert _{r,q}^{\circ }\leq
1\right\}\label{embed bidualii}
\end{equation}
Now evaluate this supremum by finding the appropriate $y$, and replacing $\left\vert y\right\vert _{r,q}^{\circ }$ with the equivalent quantity
\[
\left\vert y \right\vert_\sharp=\sup \left\{ r^{-1}k^{-1/q}\sum_{i=1}^{k}y_{[i]}:1\leq k\leq \min
\left\{ r^{q/(q-1)},n\right\} \right\} 
\]
Bounds on the coordinates of $y$ are achieved by exploiting the fact that $\left\vert y \right\vert_\sharp\leq1$ and that $y$ is a maximizer of $\sum x_iy_i$. Including non-explicit constants of the form $C_q,c_q$ may help to simplify the calculations. An alternative method is to notice that within the collection of points with positive decreasing coordinates, $\partial E_{r,q}$ is contained in a hyperplane determined by $n$ given points.
\end{proof}

\subsection{Partial reduction to the case of equal coefficients (combinatorial approach)\label{sum reduction}}

For $n\in\mathbb{N}$ and $0\leq k\leq n$, the symbol $S(n,k)$ represents the number of ways to partition a set of cardinality $n$ into a total of $k$ nonempty subsets, taking $S(n,0)=0$. This is known as a Stirling number of the second kind. It follows that the number of functions $f:\{1,2,\cdots, n\}\rightarrow \{1,2,\cdots, n\}$ with $\left\vert Range(f)\right\vert=k$ is equal to
\[
E(n,k)=\frac{n!}{(n-k)!}S(n,k)
\]
For $k\geq 1$, $S(n+1,k)=kS(n,k)+S(n,k-1)$. This can easily be seen by taking a set of $n+1$ elements, setting one aside, and considering partitions where the distinguished element stands alone as a singleton and those where it does not.

\begin{lemma}\label{Stirling no two}
For all $n\in\mathbb{N}$ and $n/2\leq k\leq n$, $k!S(n,k)\geq (n-k)!S(n,n-k)$, which can be written as $E(n,k)\geq E(n,n-k)$.
\end{lemma}

\begin{proof}
Consider the lemma as a sequence of statements $(\mathcal{S}_n)_1^\infty$. $\mathcal{S}_1$ is seen to be true since $E(1,0)=0$ while $E(1,1)=1$. Suppose that $\mathcal{S}_n$ is true for some $n\geq 1$ and consider $\mathcal{S}_{n+1}$. If $k=(n+1)/2$ then the required inequality for $\mathcal{S}_{n+1}$ holds with equality. If $k>(n+1)/2$ then $k\geq (n+2)/2$ and by $\mathcal{S}_n$, $k!S(n+1,k)=k!\left[kS(n,k)+S(n,k-1)\right]$ can be bounded below by
\begin{eqnarray*}
&&k(n-k)!S(n,n-k)+k(n-k+1)!S(n,n-k+1)\\
&\geq&(n-k+1)!S(n,n-k)+(n-k+1)(n-k+1)!S(n,n-k+1)\\
&\geq&(n-k+1)!S(n+1,n-k+1)
\end{eqnarray*}
So $\mathcal{S}_{n+1}$ is true.
\end{proof}

\begin{theorem}\label{to iid}
Let $I=(I(i))_1^n$ be an i.i.d. sequence of random integers uniformly distributed in $\{1,2,\cdots, n\}$ and $V=(V_i)_1^n$ an i.i.d. sequence of non-negative random variables independent of $I$. Then for all $b\in[0,\infty)^n$ and all $t>0$,
\[
\mathbb{P}\left\{\sum_{i=1}^n b_iV_i \geq t\right\}\leq 2\mathbb{P}\left\{\sum_{i=1}^n b_{I(i)}V_i \geq \frac{t}{2}\right\}
\]
\end{theorem}

\begin{proof}
Because the distributions in question do not depend on the underlying probability space, we may assume without loss of generality that this underlying probability space is rich enough to support the independent random variables that we introduce throughout the proof, and that it is non-atomic. Consider any $v\in[0,\infty)^n$ and let $\sigma$ be a random permutation uniformly distributed in the symmetric group $S_n$ and independent of $(I,V)$. Let
\[
E=\left\{k\in\{1,2,\cdots, n\}:\forall i\leq k, I(i)\neq I(k)\right\}   \hspace{2cm}   F=\{1,2,\cdots, n\}\setminus E
\]
Consider an i.i.d. collection of random bijections $(q_{G,H})_{P}$ indexed by
\[
P=\left\{(G,H)\in \mathcal{P}(\{1,2,\cdots, n\})\times \mathcal{P}(\{1,2,\cdots, n\}):\left\vert G\right\vert =\left\vert H\right\vert \right\}
\]
where $\mathcal{P}(\cdot)$ denotes power set and each $q_{G,H}$ is uniformly distributed among the collection of all bijections from $G$ to $H$. We do not exclude the element $(\emptyset,\emptyset)$ from $P$. We take $(q_{G,H})_{P}$ to be independent of $(I,V,\sigma)$. Let $\theta \in S_n$ be the random permutation defined as
\[
\theta(i)=\left\{
\begin{array}{ccc}
I(i)&:& i\in E\\
q_{F,\{1,2,\cdots, n\}\setminus I(E)}(i)&:& i\in F
\end{array}
\right.
\]
Note that $\theta$ is uniformly distributed in $S_n$, and independent of $(V,\sigma)$ because it is defined in terms of $I$ and $(q_{G,H})_{P}$.

\medskip

\noindent \underline{Claim 1:} $E$ is independent of $(\theta, \sigma)$.

\noindent \underline{Proof of Claim 1:} Consider any $\theta^{(0)}\in S_n$ and $E_0\subseteq \{1,2,\cdots, n\}$ with $1\in E_0$, and let $F_0=\{1,2,\cdots, n\}\setminus E_0$. Now $\left\{\theta=\theta^{(0)}\right\}\cap\left\{E=E_0\right\}$ is equal to
\begin{eqnarray*}
\left[\cap_{i\in E_0}\left\{I(i)=\theta_i^{(0)}\right\}\right]&\cap&\left[\cap_{i\in F_0}\left\{I(i)\in\left\{\theta_j^{(0)}:j<i,j\in E_0\right\}\right\}\right]\\
&\cap&\left\{q_{F_0,\{1,2,\cdots, n\}\setminus \theta^{(0)}(E_0)}=\theta^{(0)}|_{F_0}\right\}
\end{eqnarray*}
where $\theta^{(0)}|_{F_0}$ denotes the restriction of $\theta^{(0)}$ to ${F_0}$. This can be seen by showing that set inclusion holds in both directions and noting that $I(i)=I(j)$ for some $j<i$ if and only if $I(i)=I(j)$ for some $j<i$ with $j\in E$. Since $(q_{G,H})_{P}$ is independent of $I$,
\begin{eqnarray*}
&&\mathbb{P}\left(\left\{\theta=\theta^{(0)}\right\}\cap\left\{E=E_0\right\}\right)\\
&=&\mathbb{P}\left(\left[\cap_{i\in E_0}\left\{I(i)=\theta_i^{(0)}\right\}\right]\cap\left[\cap_{i\in F_0}\left\{I(i)\in\left\{\theta_j^{(0)}:j<i,j\in E_0\right\}\right\}\right]\right)\\
&\times&\mathbb{P}\left(\left\{q_{F_0,\{1,2,\cdots, n\}\setminus \theta^{(0)}(E_0)}=\theta^{(0)}|_{F_0}\right\}\right)
\end{eqnarray*}
Since the coordinates of $I$ are independent of each other this reduces to
\begin{eqnarray*}
&&n^{-\left\vert E_0\right\vert}\left(\prod_{i\in F_0}\frac{\left\vert\left\{\theta_j^{(0)}:j<i,j\in E_0\right\}\right\vert}{n}\right)\left\vert F_0\right\vert !\\
&=&n^{-n}\left(\prod_{i\in F_0}\left\vert\left\{j:j<i,j\in E_0\right\}\right\vert\right)\left\vert F_0\right\vert !
\end{eqnarray*}
Since this probability does not depend on $\theta^{(0)}$ and
\[
\sum_{\theta^*\in S_n}\mathbb{P}\left(\left\{\theta=\theta^*\right\}\cap\left\{E=E_0\right\}\right)=\mathbb{P}\left(\left\{E=E_0\right\}\right)
\]
we conclude that
\[
\mathbb{P}\left(\left\{\theta=\theta^{(0)}\right\}\cap\left\{E=E_0\right\}\right)=\frac{1}{n!}\mathbb{P}\left(\left\{E=E_0\right\}\right)=\mathbb{P}\left(\left\{\theta=\theta^{(0)}\right\}\right)\mathbb{P}\left(\left\{E=E_0\right\}\right)
\]
which is enough to show that $\theta$ and $E$ are independent. Yet $\sigma$ is independent of $\left(I,\theta\right)$ and therefore of $(E,\theta)$, so the distribution of $(E, \theta, \sigma)$ is a product measure.

\medskip

\noindent \underline{Claim 2:} For any (deterministic) $G_0,G_0^*\subseteq \{1,2,\cdots, n\}$ such that $\left\vert G_0\right\vert=\left\vert G_0^*\right\vert$, the random variables
\[
\sum_{i\in G_0}b_{\theta(i)}v_{\sigma(i)} \hspace{1cm} \text{and} \hspace{1cm}  \sum_{i\in G_0^*}b_{\theta(i)}v_{\sigma(i)}
\]
have the same distribution. Consequently, if $G_1\subseteq \{1,2,\cdots, n\}$ and $\left\vert G_0\right\vert\leq\left\vert G_1\right\vert$, then for all $t>0$,
\[
\mathbb{P}\left\{\sum_{i\in G_1}b_{\theta(i)}v_{\sigma(i)}\geq t\right\} \geq \mathbb{P}\left\{\sum_{i\in G_0}b_{\theta(i)}v_{\sigma(i)}\geq t\right\}
\]
Here we take $\sum_{i\in\emptyset}=0$.

\noindent \underline{Proof of Claim 2:} We may assume that $G_0$ and $G_0^*$ are non-empty. Consider any fixed $\omega\in S_n$ that maps $G_0$ to $G_0^*$. Then
\[
\sum_{i\in G_0}b_{\theta(i)}v_{\sigma(i)}=\sum_{i\in G_0^*}b_{\theta\omega^{-1}(i)}v_{\sigma\omega^{-1}(i)}
\]
As observed before, $\theta$ and $\sigma$ are independent and both uniformly distributed on $S_n$, so the joint distribution of $(\theta,\sigma)$ in $S_n\times S_n$ is the uniform distribution. Since $\omega$ is fixed, the same can be said of $\theta\omega^{-1}$ and $\sigma\omega^{-1}$. Yet the distributions of
\[
\sum_{i\in G_0^*}b_{\theta\omega^{-1}(i)}v_{\sigma\omega^{-1}(i)} \hspace{2cm} \text{and} \hspace{2cm} \sum_{i\in G_0^*}b_{\theta(i)}v_{\sigma(i)}
\]
are the push-forward measures of the distributions of $(\theta\omega^{-1},\sigma\omega^{-1})$ and $(\theta,\sigma)$ under the action of
\[
(\alpha,\beta)\mapsto\sum_{i\in G_0^*}b_{\alpha(i)}v_{\beta(i)}
\]
so these two sums have the same distribution. The last part of the claim follows by taking $G'\subseteq G_1$ with $\left\vert G'\right\vert=\left\vert G_0\right\vert$, using the fact that the terms are non-negative, and applying the first part of the claim to conclude that
\[
\mathbb{P}\left\{\sum_{i\in G_1}b_{\theta(i)}v_{\sigma(i)}\geq t\right\} \geq\mathbb{P}\left\{\sum_{i\in G'}b_{\theta(i)}v_{\sigma(i)}\geq t\right\} = \mathbb{P}\left\{\sum_{i\in G_0}b_{\theta(i)}v_{\sigma(i)}\geq t\right\}
\]

\medskip

\noindent \underline{Claim 3:} Let $G_0$ and $G_1$ be random subsets of $\{1,2,\cdots, n\}$, not necessarily uniformly distributed in the power set, and assume that for all $k\in\{0,1,2,\cdots, n\}$, $\mathbb{P}\left\{\left\vert G_1 \right\vert\geq k\right\} \geq \mathbb{P}\left\{\left\vert G_0 \right\vert\geq k\right\}$. Assume also that $(G_0,G_1)$ is independent of the ordered pair $(\theta,\sigma)$. Then for all $t>0$,
\[
\mathbb{P}\left\{\sum_{i\in G_1}b_{\theta(i)}v_{\sigma(i)}\geq t\right\} \geq \mathbb{P}\left\{\sum_{i\in G_0}b_{\theta(i)}v_{\sigma(i)}\geq t\right\}
\]

\noindent \underline{Proof of Claim 3:} Fix any sequence of sets $\left(G^{(k)}\right)_0^n$ with $\left\vert G^{(k)} \right\vert=k$. By independence, for any $t>0$,
\begin{eqnarray*}
\mathbb{P}\left\{\sum_{i\in G_0}b_{\theta(i)}v_{\sigma(i)}\geq t\right\}&=&\sum_{k=0}^n\sum_{\left\vert G^* \right\vert=k}\mathbb{P}\left\{\sum_{i\in G_0}b_{\theta(i)}v_{\sigma(i)}\geq t\hspace{0.2cm}\text{and}\hspace{0.2cm}G_0=G^*\right\}\\
&=&\sum_{k=0}^n\sum_{\left\vert G^* \right\vert=k}\mathbb{P}\left\{\sum_{i\in G^*}b_{\theta(i)}v_{\sigma(i)}\geq t\hspace{0.2cm}\text{and}\hspace{0.2cm}G_0=G^*\right\}\\
&=&\sum_{k=0}^n\sum_{\left\vert G^* \right\vert=k}\mathbb{P}\left\{\sum_{i\in G^*}b_{\theta(i)}v_{\sigma(i)}\geq t\right\}\mathbb{P}\left\{G_0=G^*\right\}
\end{eqnarray*}
By Claim 2 this can be written as
\[
\sum_{k=0}^n\sum_{\left\vert G^* \right\vert=k}\mathbb{P}\left\{\sum_{i\in G^{(k)}}b_{\theta(i)}v_{\sigma(i)}\geq t\right\}\mathbb{P}\left\{G_0=G^*\right\}=
\sum_{k=0}^n\mathbb{P}\left\{\sum_{i\in G^{(k)}}b_{\theta(i)}v_{\sigma(i)}\geq t\right\}\mathbb{P}\left\{\left\vert G_0 \right\vert=k\right\}
\]
Similarly,
\[
\mathbb{P}\left\{\sum_{i\in G_1}b_{\theta(i)}v_{\sigma(i)}\geq t\right\}=\sum_{k=0}^n\mathbb{P}\left\{\sum_{i\in G^{(k)}}b_{\theta(i)}v_{\sigma(i)}\geq t\right\}\mathbb{P}\left\{\left\vert G_1 \right\vert=k\right\}
\]
By Claim 2 again
\[
\mathbb{P}\left\{\sum_{i\in G^{(k)}}b_{\theta(i)}v_{\sigma(i)}\geq t\right\}
\]
is a non-decreasing function of $k$. Because the distribution of $\left\vert G_0 \right\vert$ is dominated by the distribution of $\left\vert G_1 \right\vert$, this implies that Claim 3 is true.

\medskip

\noindent \underline{Claim 4:} For all $t>0$,
\[
\mathbb{P}\left\{\sum_{i\in E}b_{\theta(i)}v_{\sigma(i)}\geq t\right\} \geq \mathbb{P}\left\{\sum_{i\in F}b_{\theta(i)}v_{\sigma(i)}\geq t\right\}
\]

\noindent \underline{Proof of Claim 4:} By Lemma \ref{Stirling no two}, the distribution of $\left\vert E \right\vert$ dominates the distribution of $\left\vert F \right\vert$. Because $E$ is independent of $(\theta, \sigma)$ and $F$ is a function of $E$, the ordered pair $(E,F)$ is independent of $(\theta, \sigma)$. Claim 4 now follows from Claim 3.

\medskip

\noindent \underline{Claim 5:} For all $t>0$,
\[
\mathbb{P}\left\{\sum_{i=1}^nb_{\theta(i)}v_{\sigma(i)}\geq t\right\} \leq 2\mathbb{P}\left\{\sum_{i=1}^nb_{I(i)}v_{\sigma(i)}\geq \frac{t}{2}\right\}
\]

\noindent \underline{Proof of Claim 5:} The LHS is bounded above by
\begin{eqnarray*}
\mathbb{P}\left\{\sum_{i\in E}b_{\theta(i)}v_{\sigma(i)}\geq \frac{t}{2}\right\}+\mathbb{P}\left\{\sum_{i\in F}b_{\theta(i)}v_{\sigma(i)} \geq \frac{t}{2}\right\}&\leq& 2\mathbb{P}\left\{\sum_{i\in E}b_{\theta(i)}v_{\sigma(i)}\geq \frac{t}{2}\right\}\\
&=&2\mathbb{P}\left\{\sum_{i\in E}b_{I(i)}v_{\sigma(i)}\geq \frac{t}{2}\right\}
\end{eqnarray*}
which is bounded above by the RHS.
\medskip

\noindent \underline{Claim 6:} The Theorem is true.

\noindent \underline{Proof of Claim 6:} Since $V$ has not entered the proof until now we can take it to be independent of everything else (assuming as we are that the underlying probability space is rich enough). So
\[
\mathbb{P}\left\{\sum_{i=1}^nb_{\theta(i)}V_{\sigma(i)}\geq t\right\}=\int_{[0,\infty)^n}\mathbb{P}\left\{\sum_{i=1}^nb_{\theta(i)}v_{\sigma(i)}\geq t\right\}d\mathbb{P}_V(v)
\]
where $\mathbb{P}_V$ is the distribution of $V$. By Claim 5 this is bounded above by
\[
\int_{[0,\infty)^n}2\mathbb{P}\left\{\sum_{i=1}^nb_{I(i)}v_{\sigma(i)}\geq \frac{t}{2}\right\}d\mathbb{P}_V(v)=2\mathbb{P}\left\{\sum_{i=1}^nb_{I(i)}V_{\sigma(i)}\geq \frac{t}{2}\right\}
\]
Now $\sum_1^nb_{I(i)}V_{\sigma(i)}=\sum_1^nb_{I\sigma^{-1}(i)}V_{i}$. Since the coordinates of $I$ are independent and uniformly distributed in $\{1,2,\cdots ,n\}$, and $\sigma$ is independent of $I$, the distribution of $(I\sigma^{-1}(i))_1^n$ is the same as that of $(I(i))_1^n$. Since $V$ is independent of $(I,\sigma)$, this then implies that the distribution of $(b_{I\sigma^{-1}(i)}V_{i})_1^n$ is the same as that of $(b_{I(i)}V_{i})_1^n$ and the theorem is proved.
\end{proof}

\subsection{Combining the geometric and combinatorial approaches}

Throughout this section we fix $n\in\mathbb{N}$ and $q\in(2,\infty)$ and consider two sequences of i.i.d. non-negative random variables $(W_i)_1^n$ and $(Y_i)_1^n$ such that for all $t>0$,
\[
\mathbb{P}\left\{W_i >t\right\}=e^{q/2}(e+t)^{-q/2}\left(\ln (e+t)\right)^{q/2}   \hspace{2cm}   \mathbb{P}\left\{Y_i >t\right\}=(1+t)^{-q/2}
\]
Let $(b_i)_1^n\in(0,\infty)^n$ and let $(I(i))_1^n$ be an i.i.d. sequence of random integers uniformly distributed in $\{1, 2, \cdots, n\}$ as in Theorem \ref{to iid}. For $\delta\in(0,1)$, let $[\cdot]_{\delta, W}$ be the norm as studied in Section \ref{dual norm} associated to the distribution of $(W_i)_1^n$ (see in particular \ref{sqrbrnoerm} and \ref{Mink fnc Poisson uncon}), and let $[\cdot]_{\delta, Y}$ be the coorresponding norm associated to the distribution of $(Y_i)_1^n$. $\left\vert b \right\vert_0=\left\vert\{i:b_i\neq 0\}\right\vert$.

\begin{proposition}\label{R1}
For all $b\in[0,\infty)^n$ and all $t>0$, with probability at least $1-Ce^{-t^2/2}$,
\begin{equation}\label{R1 linear combo}
\sum_{i=1}^nb_iW_i\leq C_q\left\vert b\right\vert _1+C_q\left(t^2+\ln \left\vert b \right\vert_0\right)e^{t^2/q}\left\vert b\right\vert_{q/2}
\end{equation}
\end{proposition}

\begin{proof}
Assume momentarily that each $b_i\neq0$. For $t>0$ let $G(t)=\mathbb{P}\left\{b_{I(i)}W_i \geq t\right\}$. By independence, and Fubini's theorem applied to $\{1,2, \cdots, n\}\times [0,\infty)$,
\begin{eqnarray*}
G(t)=\frac{1}{n}\sum_{i=1}^n\mathbb{P}\left\{W_i \geq \frac{t}{b_i}\right\}=\frac{1}{n}\sum_{i=1}^ne^{q/2}(e+tb_i^{-1})^{-q/2}\left(\ln (e+tb_i^{-1})\right)^{q/2}
\end{eqnarray*}
and
\begin{eqnarray*}
-tG'(t)=\frac{1}{n}\frac{q}{2}\sum_{i=1}^ne^{q/2}tb_i^{-1}(e+tb_i^{-1})^{-q/2-1}\left(\ln (e+tb_i^{-1})\right)^{q/2-1}\left(\ln (e+tb_i^{-1})-1\right)\leq\frac{q}{2}G(t)
\end{eqnarray*}
Now $H=G^{-1}$ is the reflected quantile function of $b_{I(i)}W_i$ and by the inverse function theorem the inequality $-tG'(t)\leq\frac{q}{2}G(t)$ can be written as $-H'(t)\geq (2/q)H(t)/t$, and then as
\[
-\frac{d}{dt}\ln H(t)\geq \frac{2}{q}t^{-1}
\]
By FTC this implies that for all $\delta, x\in (0,1)$, $H(\delta x)\geq \delta^{-2/q}H(x)$ and the assumption of Proposition \ref{final productiones} is satisfied with $p=q/2$ and $T=1$. By the conclusion of that result (see Remark \ref{remarkiou} for a simplification), with probability at least $1-Ce^{-t^2/2}$,
\[
\sum_{i=1}^nb_{I(i)}W_i\leq C_qH\left(e^{-1-2/e}\frac{1}{n+1}e^{-t^2/2}\right)+Cn\int_0^1H(x)dx
\]
The second term represents $Cn\mathbb{E}(b_{I(i)}W_i)=Cn(\mathbb{E}b_{I(i)})(\mathbb{E}W_i)=C_q\left\vert b\right\vert _1$, and we now focus on the first term. Using $\ln(e+x)\leq C_q\ln\left(e+x^{q/2}\right)$ valid for $x\geq 0$,
\begin{eqnarray*}
G(s)\leq\frac{C_q}{n}\left(\sum_{j=1}^n(e+sb_j^{-1})^{-q/2}\right)\sum_{i=1}^n\frac{(e+sb_i^{-1})^{-q/2}}{\sum_{j=1}^n(e+sb_j^{-1})^{-q/2}}\left(\ln (e+(sb_i^{-1})^{q/2})\right)^{q/2}
\end{eqnarray*}
Since $x\mapsto (\ln (e+x))^{q/2}$ for $x\in[0,\infty)$ is the same order of magnitude as a concave function (up to a factor of $C_q$), we may apply Jensen's inequality to bound this above by,
\begin{eqnarray*}
&&\frac{C_q}{n}\left(\sum_{j=1}^n(e+sb_j^{-1})^{-q/2}\right)\left[\ln \left(e+\sum_{i=1}^n\frac{(e+sb_i^{-1})^{-q/2}}{\sum_{j=1}^n(e+sb_j^{-1})^{-q/2}}(sb_i^{-1})^{q/2}\right)\right]^{q/2}\\
&\leq&\frac{C_q}{n}\left(\sum_{j=1}^n(e+sb_j^{-1})^{-q/2}\right)\left[\ln \left(e+\frac{n}{\sum_{j=1}^n(e+sb_j^{-1})^{-q/2}}\right)\right]^{q/2}
\end{eqnarray*}
If $s\geq \left\vert b\right\vert_\infty$ this is at most
\[
C_qn^{-1}s^{-q/2}\left(\sum_{i=1}^nb_i^{q/2}\right)\left[\ln \left(e+ns^{q/2}\left(\sum_{i=1}^nb_i^{q/2}\right)^{-1}\right)\right]^{q/2}
\]
Setting
\[
s=\max\left\{\left\vert b\right\vert_\infty,C_qe^{t^2/q}\left\vert b\right\vert_{q/2}\left(t^2+\ln n\right)\right\}
\]
we see that indeed $s\geq \left\vert b\right\vert_\infty$, and
\[
G(s)\leq e^{-1-2/e}\frac{1}{n+1}e^{-t^2/2}
\]
so $H\left(e^{-1-2/e}\frac{1}{n+1}e^{-t^2/2}\right)\leq s$. All of this implies that with probability at least $1-Ce^{-t^2/2}$,
\[
\sum_{i=1}^nb_{I(i)}W_i\leq C_q\left\vert b\right\vert _1+C_qe^{t^2/q}\left\vert b\right\vert_{q/2}\left(t^2+\ln n\right)
\]
By Theorem \ref{to iid} this implies that with the same probability,
\[
\sum_{i=1}^nb_iW_i\leq C_q\left\vert b\right\vert _1+C_qe^{t^2/q}\left\vert b\right\vert_{q/2}\left(t^2+\ln n\right)
\]
If $b\in[0,\infty)^n$ has exactly $j$ non-zero coordinates then we may apply this result to the truncated vector $b^*\in[0,\infty)^j$ to improve the $\ln n$ to $\ln j$, arriving at \ref{R1 linear combo}.
 \end{proof}

\begin{proposition}\label{R2}
For all $b\in[0,\infty)^n$ and all $t>0$, with probability at least $1-Ce^{-t^2/2}$,
\begin{equation}\label{R2 linear combo}
\sum_{i=1}^nb_iW_i\leq C_q\left\vert b\right\vert _1+C_qt^2e^{t^2/q}\sum_{i=1}^ni^{-1+2/q}b_{[i]}
\end{equation}
\end{proposition}

\begin{proof}
For any $1\leq k\leq n$, taking $b$ to be the vector with $1$ for its first $k$ coordinates and $0$ for the remaining $n-k$ coordinates, (\ref{R1 linear combo}) implies that with probability at least $1-Ce^{-t^2/2}$,
\[
\sum_{i=1}^nb_iW_i\leq C_qk+C_qe^{t^2/q}k^{2/q}(t^2+\ln k)
\]
For $t^2>(q/2-1)\ln k$ the term $\ln k$ can be dropped by increasing the value of $C_q$. For $t^2\leq(q/2-1)\ln k$,
\[
C_qe^{t^2/q}k^{2/q}(t^2+\ln k)\leq C_q k^{1/2+1/q}\ln k <C_qk
\]
and still the term $\ln k$ can be dropped. So the bound can be written as
\[
\sum_{i=1}^nb_iW_i\leq C_q\left(k+t^2e^{t^2/q}k^{2/q}\right)
\]
This can be written as
\[
F_b^{-1}\left(1-Ce^{-t^2/2}\right)\leq C_qk+C_qk^{2/q}t^2e^{t^2/q}
\]
Set $\delta=Ce^{-t^2/2}$. From the integral representation of $\left[ \cdot\right] _{\delta, W}$ in (\ref{Mink fnc Poisson uncon}),
\begin{eqnarray*}
\left[b\right] _{\delta, W}&\leq& C_q\delta^{-1}\int_0^1\left[k+k^{2/q}s^{-2/q}\left(\ln \frac{C}{s}\right)\right]e^{-\delta^{-1}s}ds\leq C_{q}\left( k+k^{2/q}\delta ^{-2/q}\log
\delta ^{-1}\right)\\
&=& C_q\left( \left\vert
b\right\vert _{1}+\delta ^{-2/q}\log \delta ^{-1}\left\vert b\right\vert
_{q/2}\right)\leq C_q\left\vert b\right\vert _{r,q/2}
\end{eqnarray*}%
where $r=\delta ^{-2/q}\log \delta ^{-1}$ and the last inequality follows from Lemma \ref{formuii} (i.e. the simplified formula for $\left\vert b \right\vert _{r,q/2}$ as a $\{0,1\}$-vector). Since this holds for any such $k$ and $b$, by Lemma \ref{comparing norms}, $\left[ b\right] _{\delta }\leq C_{q}\left\vert b\right\vert _{r,q/2}$ for all $b\in \mathbb{R}^{n}$. (\ref{R2 linear combo}) now follows by recalling Lemma \ref{convexity of quantiles} (that $\left[b\right] _{\delta, W}$ bounds the quantiles of $\sum b_iW_i$), and using the general estimate for $\left\vert b \right\vert _{r,q/2}$ in (\ref{mink lor fnc}).
\end{proof}

\begin{lemma}\label{comparisone}
For all $x\in\mathbb{R}^n$ with $x_1\geq x_2 \geq \cdots \geq x_n$ and $x_1\neq x_n$,
\[
\left\vert x \right\vert_{q/2}\leq\left(\frac{\left\vert x \right\vert_1-nx_n}{x_1-x_n}x_1^{q/2}+\frac{nx_1-\left\vert x \right\vert_1}{x_1-x_n}x_n^{q/2}\right)^{2/q}
\]
\end{lemma}

\begin{proof}
We maximize $f(z)=\sum_1^nz_i^{q/2}$ over the compact set $E$ of all $z\in[0,\infty)^n$ such that $\left\vert z\right\vert_1 = \left\vert x\right\vert_1$ and $x_1=z_1\geq z_2 \geq \cdots \geq z_n=x_n$. If $z\in E$ has the property that $z_i\notin\{x_1,x_n\}$ for more than one value of $i$, say $j$ and $k$ with $j<k$, we can assume (per definition) that $j$ is the least value for which this holds and $k$ is the greatest. But then there exists $\varepsilon>0$ such that $y\in E$ where
\begin{equation*}
y_i=\left\{
\begin{array}{ccc}
z_i&:& i\notin\{j,k\}\\
z_j+\varepsilon&:& i=j\\
z_k-\varepsilon&:& i=k
\end{array}
\right.
\end{equation*}
and $f(y)=(z_j+\varepsilon)^{q/2}+(z_k-\varepsilon)^{q/2}+\sum_{i\notin\{j,k\}}z_i^{q/2}>f(z)$. This follows because by convexity and comparing the slope of secant lines
\[
(z_j+\varepsilon)^{q/2}-z_j^{q/2}>z_k^{q/2}+(z_k-\varepsilon)^{q/2}
\]
By excluding such points, the maximum occurs at a point $z$ such that
\begin{equation*}
z_i=\left\{
\begin{array}{ccc}
x_1&:& i<k\\
x_n&:& i>k
\end{array}
\right.
\end{equation*}
for some $1\leq k\leq n-1$. The value of $k$ is determined by the equation $\left\vert z\right\vert_1 = \left\vert x\right\vert_1$. To re-distribute the $\ell_1^n$ mass of $x$ to form $z$, we start each coordinate at $0$, add $x_n$ to each coordinate, and then distribute the remaining total of $\left\vert x\right\vert_1-nx_n$ in doses of $x_1-x_n$ until we no longer have enough for a full dose. This implies that $k-1=\left\lfloor \alpha\right\rfloor$, where
\[
\alpha=\frac{\left\vert x \right\vert_1-nx_n}{x_1-x_n}
\]
and it follows again by convexity that
\begin{eqnarray*}
f(z)&=&\left\lfloor \alpha\right\rfloor x_1^{q/2}+\left\{(\alpha-\left\lfloor \alpha\right\rfloor)x_1+(1-(\alpha-\left\lfloor \alpha\right\rfloor))x_n\right\}^{q/2}+\left(n-1-\left\lfloor \alpha\right\rfloor\right)x_n^{q/2}\\
&\leq& \alpha x_1^{q/2}+(n-\alpha)x_n^{q/2}
\end{eqnarray*}
Since $x\in E$, $\left\vert x \right\vert_{q/2}\leq f(z)^{2/q}$.
\end{proof}

\begin{proposition}\label{R3}
For all $b\in[0,\infty)^n$ and all $t>0$, with probability at least $1-Ce^{-t^2/2}$,
\begin{equation}\label{R3 linear combo}
\sum_{i=1}^nb_iW_i\leq C_q\left(\left\vert b\right\vert _1+t^2e^{t^2/q}\left\vert b\right\vert_{q/2}\right)
\end{equation}
\end{proposition}

\begin{proof}
Note that by Proposition \ref{R1} (see the explanation about removing the $\ln k$ in the proof of Proposition \ref{R2}), the result already holds as long as
\begin{equation}
\left\vert b\right\vert_{q/2}\left\vert b\right\vert_{0}^{c_q^\#}\leq C_q^\#\left\vert b\right\vert_{1}\label{sparsity and q/2 and 1 norms}
\end{equation}
where $c_q^\#>0$ can be taken to be arbitrarily small and $C_q$ can be arbitrarily large. We fix the value of $c_q^\#$ to be the same at different appearences. Consider any $b\in[0,\infty)^n$ such that (\ref{sparsity and q/2 and 1 norms}) is violated, and assume without loss of generality that $1=b_1\geq b_2\geq b_3 \cdots \geq b_n$. Set $t=b_n$. We must have $t<1/2$ otherwise (\ref{sparsity and q/2 and 1 norms}) would hold. If $b_n>0$ then by the assumption that (\ref{sparsity and q/2 and 1 norms}) is violated and Lemma \ref{comparisone},
\[
\left\vert b\right\vert_{1}\leq  c_qn^{c_q^\#}\left(\frac{\left\vert b \right\vert_1-nt}{1-t}+\frac{n-\left\vert b \right\vert_1}{1-t}t^{q/2}\right)^{2/q}\leq c_qn^{c_q^\#}\left((\left\vert b \right\vert_1-nt)+(n-\left\vert b \right\vert_1)t^{q/2}\right)^{2/q}
\]
We now consider two cases. In Case I, $\left\vert b \right\vert_1-nt\leq (n-\left\vert b \right\vert_1)t^{q/2}$ which leads to the contradiction
\[
\frac{1}{n}\left\vert b\right\vert_{1}\leq c_q n^{-1+c_q^\#+2/q}\left(1-\frac{1}{n}\left\vert b\right\vert_{1}\right)t
\]
This is a contradiction because LHS is the average size of a coordinate while RHS is less than the smallest coordinate. In Case II, $\left\vert b \right\vert_1-nt> (n-\left\vert b \right\vert_1)t^{q/2}$ which leads to
\[
t\leq \frac{1}{n}\left(\left\vert b\right\vert_{1}-C_qn^{-c_q^\#q/2}\left\vert b\right\vert_{1}^{q/2}\right)
\]
Using $s-As^{q/2}\leq C_qA^{-2/(q-2)}$ valid for $s\geq 0$, this is bounded above by
\begin{equation}
C_qn^{-1+c_q^\#q/(q-2)}\label{decay bound after violation}
\end{equation}
This obviously holds also when $b=0$. As we argued before, the same estimate can be applied to a truncated vector of $k$ coordinates, with $n$ replaced with $k$ in this estimate, as long as the truncated vector violates (\ref{sparsity and q/2 and 1 norms}). Let $k$ be the largest integer such that the truncated vector $(b_i)_1^k$ satisfies (\ref{sparsity and q/2 and 1 norms}). Such a value of $k$ exists because every element of $\mathbb{R}^1$ satisfies (\ref{sparsity and q/2 and 1 norms}), and by our assumption that $b$ violates (\ref{sparsity and q/2 and 1 norms}), $1\leq k\leq n-1$. For all $j>k$, it follows from the definition of $k$ that $(b_i)_1^j$ violates (\ref{sparsity and q/2 and 1 norms}), so by applying (\ref{decay bound after violation}) to this vector in dimension $j$, $b_j\leq C_qj^{-1+c_q^\#q/(q-2)}$. By Propositions \ref{R1} and \ref{R2} applied to $(b_i)_1^k$ and $(b_i)_{k+1}^n$ respectively, with probability at least $1-2Ce^{-t^2/2}$,
\begin{eqnarray*}
&&\sum_{i=1}^nb_iW_i=\sum_{i=1}^kb_iW_i+\sum_{i=k+1}^nb_iW_i\\
&\leq& C_q\left(\sum_{i=1}^kb_i+t^2e^{t^2/q}\left(\sum_{i=1}^kb_i^{q/2}\right)^{2/q}+\sum_{i=k+1}^nb_i+t^2e^{t^2/q}\sum_{i=k+1}^n(i-k)^{-1+2/q}b_i\right)
\end{eqnarray*}
The reason for $(i-k)$ is that for $i\geq k+1$, $b_i$ is the $(i-k)^{th}$ coordinate of $(b_i)_{k+1}^n$. By our estimate on $b_j$ for $j>k$,
\[
\sum_{i=k+1}^n(i-k)^{-1+2/q}b_i=\sum_{i=1}^{n-k}i^{-1+2/q}b_{i+k}\leq C_q\sum_{i=1}^{n-k}i^{-1+2/q}(i+k)^{-1+c_q^\#q/(q-2)}
\]
Consider the case where $n\geq 2k$ (the case $n<2k$ is similar, just with one less term). This is bounded above by
\begin{eqnarray*}
C_qk^{-1+c_q^\#q/(q-2)}\sum_{i=1}^{k}i^{-1+2/q}+C_q\sum_{i=k+1}^{\infty}i^{-2+2/q+c_q^\#q/(q-2)}\leq C_qk^{-1+2/q+c_q^\#q/(q-2)}
\end{eqnarray*}
Since $q>2$ we can choose $c_q^\#>0$ so that $-1+2/q+c_q^\#q/(q-2)<0$. The bound on $\sum_1^nb_iW_i$ then becomes
\[
C_q\left(\left\vert b\right\vert_1+t^2e^{t^2/q}(1+\left\vert b\right\vert_{q/2})\right)
\]
and the $1$ can be deleted by our assumption that $b_1=1$.
\end{proof}

\begin{proposition}\label{R4}
For all $b\in[0,\infty)^n$ and all $t>0$, with probability at least $1-Ce^{-t^2/2}$,
\begin{equation}\label{R4 linear combo}
\sum_{i=1}^nb_iY_i\leq C_q\left(\left\vert b\right\vert _1+e^{t^2/q}\left\vert b\right\vert_{q/2}\right)
\end{equation}
\end{proposition}

\begin{proof}
The proof is almost identical to that of Proposition \ref{R1}, but simpler because it does not involve the logarithmic term. Setting
\[
G(t)=\mathbb{P}\left\{b_{I(i)}Y_i\geq t\right\}=\frac{1}{n}\sum_{i=1}^n\mathbb{P}\left\{Y_i\geq\frac{t}{b_i}\right\}=\frac{1}{n}\sum_{i=1}^n\left(1+tb_i^{-1}\right)^{-q/2}
\]
This function satisfies the conditions of Proposition \ref{final productiones} with $p=q/2$, $j=0$ and $k=n-1$, the reasoning for this is the same as in the proof of Proposition \ref{R1}. So applying the simplified bound as in Remark \ref{remarkiou}, with probability at least $1-Ce^{-t^2/2}$,
\begin{eqnarray*}
\sum_{i=1}^nb_{I(i)}Y_i&\leq& C\left(1+t^{-4/q}e^{-t^2/q}\min\{t^2,n\}\right)H\left(e^{-1-2/e}\frac{1}{n+1}e^{-t^2/2}\right)+C\left\vert b\right\vert_1 \mathbb{E}Y_1\\
&\leq&C\left(1+t^{-4/q}e^{-t^2/q}\min\{t^2,n\}\right)e^{t^2/q}\left\vert b\right\vert_{q/2}+C\left\vert b\right\vert_1 \mathbb{E}Y_1
\end{eqnarray*}
where $H=G^{-1}$. By Theorem \ref{to iid} the result can be transferred to $\sum_{i=1}^nb_iY_i$.
\end{proof}

\section{Proof of Theorem \ref{poly tails Lip}}

First we assume that each $X_i$ has a distribution that is symmetric about $0$, and then we may assume without loss of generality that each $a_i\geq 0$, and that $\left\vert a\right\vert=1$. We will use a standard trick in analysis of introducing random signs. Let $(\varepsilon_i)_1^n$ be an i.i.d. sequence of Rademacher random variables, i.e. each $\varepsilon_i$ takes the values $\pm 1$ each with probability $1/2$, independent of $(X_i)_1^n$. By the assumed independence and symmetry, the vector $(\varepsilon_i\left\vert X_i\right\vert)_1^n$ has the same distribution as $(X_i)_1^n$, and
\begin{eqnarray*}
\mathbb{P}\left\{\left\vert \sum_{i=1}^na_iX_i\right\vert >t\right\}&=&\mathbb{P}\left\{\left\vert \sum_{i=1}^na_i\varepsilon_i\left\vert X_i\right\vert\right\vert >t\right\}=\int_{[0,\infty)^n}\mathbb{P}\left\{\left\vert \sum_{i=1}^nx_i\varepsilon_i\right\vert >t\right\}d\mathbb{P}_{\left(a_i\left\vert X_i\right\vert\right)_1^n}(x)
\end{eqnarray*}
where $\mathbb{P}_{\left(a_i\left\vert X_i\right\vert\right)_1^n}$ is the distribution of $\left(a_i\left\vert X_i\right\vert\right)_1^n$. Since $(\varepsilon_i)_1^n$ are independent and sub-Gaussian, with universal constants,
\[
\mathbb{P}\left\{\left\vert \sum_{i=1}^nx_i\varepsilon_i\right\vert >t\right\}\leq C\exp\left(\frac{-ct^2}{\left\vert x \right\vert^2}\right)
\]
So, using Proposition \ref{R4}, the probability a few lines above is at most
\begin{eqnarray*}
C\mathbb{E}\exp\left(\frac{-ct^2}{\sum_{1}^n a_i^2X_i^2} \right)\leq C\int_0^\infty ue^{-u^2/2}\exp\left(\frac{-ct^2}{C_q\sum a_i^2+C_qe^{u^2/q}\left(\sum a_i^q\right)^{2/q}}\right)du
\end{eqnarray*}
Here we are using the fact that the tail probabilities of $X_i^2$ decay as $t^{-q/2}$. This integral splits into two, the first one being
\[
C\int_0^{\sqrt{2q\ln(\left\vert a\right\vert/\left\vert a\right\vert_q)}} ue^{-u^2/2}\exp\left(\frac{-ct^2}{C_q\sum a_i^2}\right)du\leq Ce^{-c_qt^2}
\]
and the second one being
\[
C\int_{\sqrt{2q\ln(\left\vert a\right\vert/\left\vert a\right\vert_q)}}^\infty ue^{-u^2/2}\exp\left(\frac{-ct^2}{C_qe^{u^2/q}\left(\sum a_i^q\right)^{2/q}}\right)du=C_qt^{-q}\left\vert a \right\vert_q^{q}\int_0^{C_qt^2\left\vert a \right\vert^{-2}} e^{-\omega}\omega^{-1+q/2}d\omega
\]
where we have set $\omega=c_qt^2e^{-u^2/q}\left\vert a \right\vert_q^{-2}$. This implies
\[
\mathbb{P}\left\{\left\vert \sum_{i=1}^na_iX_i\right\vert >t\right\}\leq Ce^{-c_qt^2}+C_qt^{-q}\left\vert a \right\vert_q^{q}
\]
which can be written as
\[
\mathbb{P}\left\{\left\vert \sum_{i=1}^na_iX_i\right\vert >C_q\left(t\left\vert a \right\vert+e^{t^2/(2q)}\left\vert a \right\vert_q\right)\right\}\leq Ce^{-t^2/2}
\]
and therefore
\[
\mathbb{P}\left\{\left\vert \sum_{i=1}^na_iX_i-\mathbb{M}\sum_{i=1}^na_iX_i\right\vert >C_q\left(t\left\vert a \right\vert+e^{t^2/(2q)}\left\vert a \right\vert_q\right)\right\}\leq Ce^{-t^2/2}
\]
When $X_i$ are no longer assumed to be symmetric, apply what has been proved to $X_i-X_i'$, where $(X_i')_1^n$ is an independent copy of $(X_i')_1^n$, so
\begin{equation}
\mathbb{P}\left\{\left\vert \left(\sum_{i=1}^na_iX_i-\mathbb{M}\sum_{i=1}^na_iX_i\right)-\left(\sum_{i=1}^na_iX_i'-\mathbb{M}\sum_{i=1}^na_iX_i'\right)\right\vert >C_q\left(t\left\vert a \right\vert+e^{t^2/(2q)}\left\vert a \right\vert_q\right)\right\}\label{prob in quest}
\end{equation}
is at most $Ce^{-t^2/2}$. However, by independence, this probability can be expressed as the integral over $\mathbb{R}^n$ of the function that maps $x$ to
\[
\mathbb{P}\left\{\left\vert \left(\sum_{i=1}^na_iX_i-\mathbb{M}\sum_{i=1}^na_iX_i\right)-\left(\sum_{i=1}^na_ix_i-\mathbb{M}\sum_{i=1}^na_iX_i'\right)\right\vert >C_q\left(t\left\vert a \right\vert+e^{t^2/(2q)}\left\vert a \right\vert_q\right)\right\}
\]
Integration is performed with respect to ${P}_X$, the distribution of $(X_i)_1^n$. However, setting
\[
E=\left\{x\in\mathbb{R}^n:\sum_{i=1}^na_ix_i\leq\mathbb{M}\sum_{i=1}^na_iX_i'\right\}
\]
we see that ${P}_X(E)\geq 1/2$ and for all $x\in E$,
\begin{eqnarray*}
&&\mathbb{P}\left\{\left(\sum_{i=1}^na_iX_i-\mathbb{M}\sum_{i=1}^na_iX_i\right)-\left(\sum_{i=1}^na_ix_i-\mathbb{M}\sum_{i=1}^na_iX_i'\right) >C_q\left(t\left\vert a \right\vert+e^{t^2/(2q)}\left\vert a \right\vert_q\right)\right\}\\
&\geq&\mathbb{P}\left\{\sum_{i=1}^na_iX_i-\mathbb{M}\sum_{i=1}^na_iX_i >C_q\left(t\left\vert a \right\vert+e^{t^2/(2q)}\left\vert a \right\vert_q\right)\right\}
\end{eqnarray*}
So the probability in \ref{prob in quest} is at most $Ce^{-t^2/2}$ and at least
\[
\frac{1}{2}\mathbb{P}\left\{\sum_{i=1}^na_iX_i-\mathbb{M}\sum_{i=1}^na_iX_i >C_q\left(t\left\vert a \right\vert+e^{t^2/(2q)}\left\vert a \right\vert_q\right)\right\}
\]
so
\[
\mathbb{P}\left\{\sum_{i=1}^na_iX_i-\mathbb{M}\sum_{i=1}^na_iX_i >C_q\left(t\left\vert a \right\vert+e^{t^2/(2q)}\left\vert a \right\vert_q\right)\right\}\leq Ce^{-t^2/2}
\]
A similar argument implies
\[
\mathbb{P}\left\{\sum_{i=1}^na_iX_i-\mathbb{M}\sum_{i=1}^na_iX_i <-C_q\left(t\left\vert a \right\vert+e^{t^2/(2q)}\left\vert a \right\vert_q\right)\right\}\leq Ce^{-t^2/2}
\]
and these last two estimates imply \ref{poly conc one} with $\mathbb{M}$ instead of $\mathbb{E}$, but that deviation inequality gives a bound on the distance between $\mathbb{M}$ and $\mathbb{E}$, and using the triangle inequality we can then replace $\mathbb{M}$ with $\mathbb{E}$ (this is very standard).

\end{document}